\documentclass{amsart}
\usepackage{amssymb,latexsym}
\theoremstyle{plain}
\newtheorem{theorem}{Theorem}
\newtheorem{corollary}{Corollary}

\newtheorem{lemma}{Lemma}
\theoremstyle{definition}

\newtheorem{example}{Example}

\newtheorem{remark}{Remark}
\newtheorem{question}{Question}

\usepackage{color}
\usepackage{graphicx}

\DeclareMathOperator{\Res}{Res}
\DeclareMathOperator{\red}{red}

\newcommand{\enm}[1]{\ensuremath{#1}}          %

\newcommand{\cal}[1]{\mathcal{#1}}

\newcommand{\NN}{\enm{\mathbb{N}}}

\newcommand{\ZZ}{\enm{\mathbb{Z}}}

\newcommand{\PP}{\enm{\mathbb{P}}}

\newcommand{\Ii}{\enm{\cal{I}}}

\newcommand{\Oo}{\enm{\cal{O}}}

\newcommand{\Uu}{\enm{\cal{U}}}

\usepackage[latin1]{inputenc}
\date{}

\begin{document}

\title[Range A]
{The maximal genus of space curves in the Range A}
\author{Edoardo Ballico and Philippe Ellia}
\address[Edoardo Ballico]{Dipartimento di Matematica, Universit\`a di Trento, Via Sommarive 14, 38123 Povo (Trento), Italy.}
\email{edoardo.ballico@unitn.it}
\address[Philippe Ellia]{Dipartimento di Matematica e Informatica, Universit\`a degli Studi di Ferrara, Via Machiavelli 30, 44121 Ferrara, Italy.} 
\email{phe@unife.it}
\thanks{Partially supported by MIUR and GNSAGA of INdAM (Italy)}

\begin{abstract}
Fix positive integers $d, m$ such that $\frac{m^2+4m+6}{6} \le d < \frac{m^2+4m+6}{3}$ (the so-called Range A for space curves). Let $G(d,m)$ be the maximal genus of a smooth and connected degree $d$ curve $C\subset \PP^3$ such that $h^0(\Ii _C(m-1)) =0$. Here we prove that $G(d,m) = 1+(m-1)d -\binom{m+2}{3}$ if $m\ge 13.8\cdot 10^5$. The case $\frac{m^2+4m+6}{4} \le d < \frac{m^2+4m+6}{3}$ was known by work of Fl{\o}ystad and joint work of Ballico, Bolondi, Ellia, Mir\`{o}-Roig. To prove the case
 $\frac{m^2+4m+6}{6} \le d < \frac{m^2+4m+6}{4}$ we show that in this range for large $d$ every integer between $0$ and  $1+(m-1)d -\binom{m+2}{3}$ is the genus
 of some degree $d$ smooth and connected curve $C\subset \PP^3$ such that $h^0(\Ii _C(m-1)) =0$.
\end{abstract}

\maketitle

\section{Introduction}

Fix integers $m \ge 2$ and $d\ge 3$. Let $G(d,m)$ be the maximal genus of a smooth and connected curve, of degree $d$, $C\subset \PP^3$ with $h^0(\Ii _C(m-1))=0$. A classical problem which goes back to Halphen \cite{Halp}, is the computation of
the integer $G(d,m)$ (\cite[Problem 3.1]{h0}, \cite{h1}).  For this problem the set of all $(d,m)\subset \NN^2$ was divided in the following $4$ regions (ranges $\emptyset$, A, B and C) (\cite{bbem, fl1, h0, h1}), because both the integer $G(d,m)$ and the geometric properties of the curve with maximal genera are very different in different regions. Set 
$$G_A(d,m):= 1+(m-1)d -\binom{m+2}{3}.$$
If $d < \frac{m^2+4m+6}{6}$, then no such curve exists (\cite[Theorem 3.3]{h0}). Hence this is called the Range $\emptyset$.
Range A is when
\begin{equation}\label{eqi1}
\frac{m^2+4m+6}{6} \le d < \frac{m^2+4m+6}{3}
\end{equation}

In Range A it is easy to see that $G(d,m) \le G_A(d,m)$ (\cite[Theorem 3.3]{h0}) and it was conjectured that equality holds
(\cite[page 364]{hh1}). This conjecture is known to be true for all $d,m$ such that $\frac{m^2+4m+6}{4} \le d <
\frac{m^2+4m+6}{3}$ (\cite{fl1}, \cite[Corollary 2.4]{bbem}). We also mention \cite{fl2} and \cite{hh1} which settle a few
cases just on the right of the last inequality of (\ref{eqi1}). In this paper we prove the following result.

\begin{theorem}\label{i1}
Fix an integer $m\ge 13.8\cdot 10^5$. Let $d$ be any integer satisfying (\ref{eqi1}). Then $G(d,m) = G_A(d,m)$.
\end{theorem}
Fix $m$ and $d$ in the Range A, i.e. satisfying (\ref{eqi1}). By \cite[first part of the proof of Theorem 3]{h0} any curve $C\subset \PP^3$ with degree $d$ and genus $g$ with $g= G_A(d,m)$ has $h^1(\Oo _C(m-1)) =0$, i.e. $h^2(\Ii _C(m-1))=0$. Since $h^0(\Ii _C(m-1))=0$ and $(m-1)d+1 -g =\binom{m+2}{3}$, Riemann-Roch gives $h^0(\Oo _C(m-1)) =\binom{m+2}{3}$ and hence
$h^i(\Ii _C(m-1)) =0$ for all $i\in \NN$. Thus the curves proved to exist in Theorem \ref{i1} have some nice cohomological properties. Now we add the stronger assumption that $h^1(\Oo _C(m-2)) =0$, i.e. $h^2(\Ii _C(m-2)) =0$. Since $\dim C =1$, the exact sequence $$0 \to \Ii _C(t)\to \Oo _{\PP^3}(t) \to \Oo_C(t)\to 0$$gives $h^i(\Ii _C(m-i)) =0$ for all $i\ge 3$. Thus the 
Castelnuovo-Mumford's lemma gives $h^1(\Ii _C(t)) =0$ for all $t\ge m$, that $\Ii _C(m))$ is globally generated and that the
homogeneous ideal of $C$ is generated by forms of degree $m$ (\cite[p. 99]{m}, \cite[\S 3]{bm}). Thus with this additional
assumption we would get strong geometrical properties of $C$. In the next theorem we will prove that $h^1(\Oo _C(m-2))=0$ and
hence the homogeneous ideal of $C$ is generated by degree $m$ forms.

For any scheme $X\subset \PP^3$ let $N_X$ denote its normal sheaf.  Our main result is the following one.

\begin{theorem}\label{i2}
Fix positive integers $m$ and $d$ such that $\frac{m^2+4m+6}{6} \leq d < \frac{m^2+4m+6}{4}$ and
$G_A(d,m) \ge 0.34\cdot 10^{15}$. Then
$G(d,m) = G_A(d,m)$ and there is a smooth and connected curve $C\subset \PP^3$ of degree $d$ and genus $G(d,m)$ such that
$h^1(\Oo _C(m-2)) =0$, $h^i(\Ii _C(m-1)) =0$, $i=0,1$, and
$h^1(N_C(-1))=0$.
\end{theorem}

We will get Theorem \ref{i1} from Theorem \ref{i2} with a small argument.

In the statement of Theorem \ref{i2} we assumed that $d<\frac{m^2+4m+6}{4}$, because the range $\frac{m^2+4m+6}{4} \le d < \frac{m^2+4m+6}{3}$
is covered by \cite[Proposition 2.2 and Corollary 2.4]{bbem}. In the range of \cite[Corollary 2.4]{bbem} our proof of Theorem \ref{i2} is very bad (it gives examples of nice curves $C$, but not enough to cover all $d$). We stated in Theorem \ref{i2} that the solution $C$ satisfies $h^1(N_C(-1)) =0$, because this vanishing has the following
interesting geometrical consequences. Since $h^1(N_C(-1)) =0$, we have $h^1(N_C)=0$ and hence the Hilbert scheme $\mathrm{Hilb}(\PP)^3$ is smooth and of dimension $4d$ at $[C]$ (\cite[\S 1]{pe}). Let $\Gamma$ be the unique irreducible component of $\mathrm{Hilb}(\PP)^3$ containing $[C]$. Fix a a plane $H\subset \PP^3$. Since $h^1(N_C(-1)) =0$, for a general $S\subset H$ with cardinality $d$ there is $X\in \Gamma$ such that $X\cap \Gamma =S$ (\cite{l}, \cite[Theorem 1.5]{pe}).

Take $(d,m)$ in the Range A and an integer $g$ such that $0\le g < G_A(d,m)$. Is there a smooth and connected curve $C\subset \PP^3$ of degree $d$ and genus $g$ with $h^0(\Ii _C(m-1)) =0$~? In the upper half of Range A we know it only if $G_A(d,m) -g$ is small (\cite[Proposition 4.3]{bbem}). In the lower half of the Range A we adapt the proof of Theorem \ref{i2} to prove the following result.

\begin{theorem}\label{i3}
Fix integers $m, d, g$ such that $m\ge 13.8\cdot 10^5$, $\frac{m^2+4m+6}{6} \leq d < \frac{m^2+4m+6}{4}$ and $0\le g \le G_A(d,m)$. Then there is a smooth and connected
curve $C\subset \PP^3$ of degree $d$ and genus $g$ such that $h^0(\Ii _C(m-1)) =0$ and $h^1(N_C(-1)) =0$.
\end{theorem}

As an easy corollary of Theorem \ref{i3} we get the following result.

\begin{corollary}\label{i5}
Fix integers $d, m$ such that $m\ge 13.8\cdot 10^5$ and $d \ge \frac{m^2+4m+6}{4}$. Set $\delta := \lfloor \frac{m^2+4m+5}{4}\rfloor$. Fix an integer $g$ such that
$0\le g \le G_A(\delta ,m)$. Then there is a smooth and connected
curve $C\subset \PP^3$ of degree $d$ and genus $g$ such that $h^0(\Ii _C(m-1)) =0$ and $h^1(N_C(-1)) =0$.
\end{corollary}

There are a few natural questions related to the topic of this paper.  Fix positive integers $d, m$.  We recall that if  $d >
m(m-1)$ (the so-called Range C) we have $G(d,m) = 1 + [d(d+m^2-4m) -r(m-r)(m-1)]/2m$, where $r$ is the only integer such that
$0\le r
\le m-1$ and $d+r\equiv 0\pmod{m}$, and equality holds if and only if the curve is linked to a plane curve of degree $r$ by the
complete intersection of a surface of degree $m$ and a surface of degree $d+r$ (\cite{gp1}). For $d\gg m$ we have $G_C(d,m)
\sim d^2/2m$ and hence $G(d,m) =G_C(d,m)\gg G_A(d,m)$. Now assume
$$\frac{m^2+4m+6}{3}\le d   \le m^2-m$$ (the Range B).  R. Hartshorne and A. Hirschowitz proved that $G(d,m) \ge G_B(d,m)$, where $G_B(d,m)$ is a complicated explicit function (\cite[Th. 5.4]{hh1}); the integer $G_B(d,m)-G_A(d,m)$ is clearly described as a sum of two terms in \cite[Th. 5.4]{hh1}. For this, using reflexive sheaves with prescribed Hilbert function (\cite{h2, hi1}), they constructed curves $C$ achieving this bound. Moreover they conjectured that $G(d,m) =G_B(d,m)$ in this range. There are some partials results (\cite{h2},
\cite{e3, h2,es,str1,str2}), but this conjecture is still widely open.

\begin{question}\label{qo1}
Take $d, m$ in the Range B. Is every integer $g$ such that $0\le g \le G_B(d,m)$ the genus of a smooth curve $C\subset \PP^3$ such that $\deg ({C})=d$ and $h^0(\Ii _C(m-1)) =0$?
\end{question}

\begin{remark}
Theorem \ref{i3} says that for the pair $(d,m)$ listed in its statement all the genera up to the maximal one are
realized. In the range C  for the pairs $(d,m)$ very near to the maximum genus $G(d,m)$ there are well-known gaps (called
Halphen's gaps) for the genera of degree $d$ smooth space curve $C$ with $h^0(\Ii _C(m-1)) =0$ (\cite{e2}, \cite[Theorem
3.3]{d}). If we use the unknown integer $G(d,m)$ instead of $G_B(d,m)$, then Question \ref{qo1} asks a proof of the
non-existence of Halphen's gaps in the Range B. We also conjecture that Halphen's gaps do not arise in the Range A.
We conjecture that Question \ref{qo1} has a positive answer, except at most for a tiny set of pairs $(d,m)$ near the Range
C. For the Range C we conjecture that all $(d,g,m)$ with $(d,m)$ in the Range C and $0\le g\le G_C(d,m+1)$ are realized by
some smooth curve $C\subset \PP^3$.
\end{remark}

We explained before and after Theorem \ref{i2} the geometric properties (Hilbert function, Hilbert function of its general hyperplane section and the smoothness of the Hilbert scheme)
satisfied by any curve $C$ as in the statement of Theorem \ref{i2}. All curves $C$ realizing $G(d,s)$ in the range C are
arithmetically Cohen-Macaulay and in particular they have maximal rank and the Hilbert scheme $\mathrm{Hilb}(\PP^3)$ is
smooth at $[C]$ of known dimension $h^0(N_C)$ (\cite{el}). In almost all cases $h^0(N_C) >4d$. In
general we may ask:

\begin{question} Is it true that any curve, $C$, of degree $d$, genus $G(d,s)$, with $h^0(\Ii _C(s-1))=0$, is of maximal rank and is a smooth point of the Hilbert
scheme~?

For given $d,s$, do these curves belong to the same irreducible component of the Hilbert scheme~?
\end{question}

See \cite{d2} for partial results.
In section \ref{Srod} we briefly describe the proof, describe the main novelties of this paper, e.g. the difference with \cite{bef} and the main numerical obstructions arising for small $m$. The
proofs of our theorems use an inductive proof, called the Horace Method, first used by A. Hirschowitz (\cite{hi}).

We thank the referee for many useful comments and suggestions.

\section{A roadmap}\label{Srod}
In section \ref{Spr} we describes curves $C_t \subset \PP^3$ and the union $C_{t,k}$, $t\ge k$, of a curve $C_t$ and a curve $C_k$. Their main properties is that $h^i(\Ii _{C_{k,t}}(k+t-1)) =0$ for all $i\ge 0$. Curves $C_t$ are used in \cite{bbem,e,fl1,fl2}. The curves $C_{t,t}$ and $C_{t,t-1}$ are the starting point of the inductive proof of the existence of curves with maximal rank given in \cite{bef} for ranges of pairs $(d,g)$ with $g$ of order $d^{3/2}$, i.e. very far from the Brill-Noether range. In the present paper we use all curves $C_{t,k}$ with, say, $t\le 200k$. In Example \ref{exk1} we explain why a key inductive step  does not work if, for a fixed $k$, we try to use the curves $C_{t,k}$ with $t\gg k$. 

Then we add inductively non-special curves in the following way. For all integers $s\ge k+t+1$ such that $s\equiv k+t-1
\pmod{2}$ we construct a non-special curve $Y_s\subset \PP^3$ with a certain degree $a(s,t,k)$ and genus $g(s,t,k)$ such that
$Y_s\cap C_{t,k} =\emptyset$ and $h^i(\Ii _{C_{t,k}\cup Y_s}(s)) =0$ for all $i\ge 0$. Setting $g_t:= h^1(\Oo _{C_t})$ and
$g_{t,k}:= g_t+g_k$ the smooth curve $C_{t,k}\cup Y_s$ has $3$ connected components and $h^1(\Oo _{C_{t,k}\cup Y}) = g_{t,k}
+g(s,t,k)$. Since $\lim _{s\to +\infty} g(s,t,k) =+\infty$ (Lemma \ref{oov2}), for each fixed genus $g$ we may find $s\gg 0$
with $g_{t,k}+g(s,t,k)>g$. 

The novel part of this paper are essentially sections  \ref{Sg} and \ref{Se} (or, rather that section \ref{Sg} allows us to conclude the proof of all theorems in section \ref{Se}).
We need to prove the theorems of the introduction for a fixed very large $m$. We fix $m$ and $d$ and take $g:= G_A(d,m)$. For large $d, m$ we may assume that $g\ge g_{1000,1000}$. To be sure that for large $m$
and hence (since we are in the Range A a large $d$) we may take a very large
$g$ ($g$ with order $m^3\gg d$) we use \cite{bef}. We take $t, k $ such that $g_{t,k} \le g \le g_{t,k+2}$ and $t+k \equiv m
\pmod{2}$. Then we get a maximal integer $y\equiv t+k-1 \pmod{2}$ such that $g_{t,k}+g(y,t,k) \le g$. A key numerical lemma is
that for a large
$m$ we have $y\le m-7$ (Lemma \ref{ov9}). Then for $x\ge y+2$ and $x\equiv y \pmod{2}$ we get a smooth connected curve $X_x$
such that $X_x\cap C_{t,k} = \emptyset$, $p_a(X_sg)+g_{t,k} =g$ and $h^1(\Ii _{C_{t,k}\cup X_x}(x))=0$  (Assertion $B(x,t,k)$ and its proof in Lemmas \ref{ov12} and \ref{ov11}).  We use a modification $B'(m-3,t,k)$ of
$B(m-3,t,k)$ to get in one step a smooth and connected curve $X$ such that $h^i(\Ii _X(m-1)) =0$ for
all $i$ and $X$ has degree $d$ and genus $g = G_A(d,m)$. Small modifications of this last step give all theorems stated in the
introduction.

In the construction we allow a parameter $\alpha$ and so the reader will met $a(s,t,k)_\alpha$ and $g(s,t,k)_\alpha$. From section \ref{Sg} on we just take $\alpha:= 202$
and the integers $a(s,t,k)$ and $g(s,t,k)$ are just the integers $a(s,t,k)_\alpha$ and $g(s,t,k)_\alpha$ with $\alpha =202$.

\section{Preliminaries}\label{Spr}
We work over an algebraically closed base field with characteristic zero. By \cite[Corollary 2.4]{bbem} we only look at $(d,m)$ in the Range A and with $d< (m^2+4m+6)/4$.

For any $o\in \PP^3$ let $2o$ denote the zero-dimensional subscheme of $\PP^3$ with $(\Ii _o)^2$ as its ideal sheaf. The scheme $2o\subset \PP^3$ is a zero-dimensional
scheme with $\deg (2o) =4$ and $2o_{\red}= \{o\}$.
Let $Q\subset \PP^3$ be a smooth quadric. For any scheme $X\subset \PP^3$ the residual scheme $\Res _Q(X)$ of $X$ with respect to $Q$ is the closed subscheme
of $\PP^3$ with $\Ii _X:\Ii _Q$ as its ideal sheaf. For any $t\in \ZZ$ we have an exact sequence (often called the residual sequence of $X$ and $Q$):
$$0\to \Ii _{\Res _Q(X)}(t-2)\to \Ii _X(t)\to \Ii _{X\cap Q,Q}(t)\to 0$$
Now assume $o\in X_{\red}\cap Q$ and the existence of a neighborhood $\Uu$ of $o$ in $\PP^3$ such that $X\cap \Uu = 2o\cup T$ with $T$ a nodal curve, $T\subset Q$,
and $o$ a singular point of $T$. Then $\Res _Q(X)\cap \Uu = \{o\}$. See \cite{hi} for many pictures, which explain how to use the Horace lemma with respect to
$Q$. We do a similar proof in several key lemmas (with $T$ a union of  $e\le 202$ lines in a ruling of $Q$ and $\delta
-e$ lines in the other ruling of $Q$).

Fix an integer $t>0$. Let $C_t \subset \mathbb {P}^3$ denote any curve fitting in an exact sequence
\begin{equation}\label{eqa6}
0 \to t\mathcal {O} _{\mathbb {P}^3}(-t-1) \to (t+1)\mathcal {O} _{\mathbb {P}^3}(-t) \to \mathcal {I} _{C_t} \to 0
\end{equation}
(\cite{e}). Each $C_t$ is arithmetically Cohen-Macaulay and in particular $h^0(\mathcal {O} _{C_t})=1$. By taking the Hilbert function in (\ref{eqa6}) we
get $\deg (C_t) = t(t+1)/2$, $p_a(C_t) = 1+ t(t+1)(2t-5)/6$, $h^1(\Oo _{C_t}(t-2)) >0$ and $h^1(\Oo _C(t-1)) =0$. Hence $h^i(\mathcal {I}_{C_t}(t-1)) =0$, $i=0,1,2$. Note that $p_a(C_t)=1+\deg (C_t)(t-1)-\binom{t+2}{3} $.
The set of all curves $C_t$ fitting in (\ref{eqa6}) is an irreducible variety and its general member is smooth and connected.
Among them there are the stick-figures called $\mathbf{K}_t$ in \cite{fl1}, \cite{fl2} and \cite{bbem} (see \cite[Lemma 2.11]{fl1}), but we only use the smooth $C_t$. Since $h^1(N_{C_t}(-2)) =0$ (\cite[proof of Proposition 2]{e}, \cite{fl1}, \cite[page 4592]{bbem}), 
for a general $S\subset Q$ with $\sharp (S)=t(t+1)$ we may find $C_t$ with $C_t\cap Q=S$ (\cite[Theorem 1.5]{pe}).
Set $C_{t,0}:= C_t$. For all positive integers $t,k$ let $C_{t,k}$ denote any disjoint union of a curve $C_t$ and a curve $C_k$. If $t>0$ and $k>0$,
then $\deg (C_{t,k}) =t(t+1)/2 +k(k+1)/2$, $h^0(\Oo _{C_{t,k})}=2$ and $h^1(\Oo _{C_{t,k}}) =2+t(t+1)(2t-5)/6 +k(k+1)(2k-5)/6$. Since $h^1(N_{C_{t,k}}(-2)) =0$,
for a general $S\subset Q$ with $\sharp (S)=t(t+1)+k(k+1)$ there is some $C_{t,k}$ with $C_{t,k}\cap Q = S$. Equivalently, for a general $C_{t,k}$
the set $C_{t,k}\cap Q$ is general in $Q$. We have $h^i(\Ii _{C_{t,k}}(t+k-1)) =0$ (\cite{bef}). For reader's sake we reproduce the statement and the proof from \cite{bef}.

\begin{lemma}\label{b1+++++}
(\cite[Lemma 2.1]{bef}) We have $h^i(\mathcal {I} _{C_{t,k}}(t+k-1)) =0$, $i=0,1,2$.
\end{lemma}

\begin{proof}
(\cite[Lemma 2.1]{bef}) Since $C_t\cap C_k =\emptyset$, we have $Tor^{1}_{{\mathcal {O} _{\mathbb {P}^3}}}(\mathcal {I}_{C_t},\mathcal {I} _{C_k})=0$ and $\mathcal {I} _{C_t}\otimes \mathcal {I} _{C_k} = \mathcal {I}_{C_{t,k}}$.
Therefore tensoring (\ref{eqa6}) with $\mathcal {I} _{C_k}(t+k-1)$ we get
\begin{equation}\label{eqa7++++}
0 \to t\mathcal {I} _{C_k}(k-2) \to (t+1)\mathcal {I} _{C_k}(k-1)\to \mathcal {I} _{C_{t,k}}(t+k-1)\to 0
\end{equation}
We have $h^2(\mathcal {I} _{C_k}(x)) = h^1(\mathcal {O} _{C_k}(x))$ and the latter integer is zero for all integers $x\ge k-2$, because $k-3$ is the maximal integer $z$ with
$h^1(\Oo _{C_k}(z)) >0$. We have $h^1(\mathcal {I} _{C_k}(k-1))=0$, because $C_k$
is arithmetically Cohen-Macaulay. Taking $k$ instead of $t$ in (\ref{eqa6}) we get $h^0(\mathcal {I} _{C_k}(k-1)) =0$. Hence (\ref{eqa7++++}) gives
$h^i(\mathcal {I} _{C_{t,k}}(t+k-1)) =0$, $i=0,1,2$.
\end{proof} 

If $t>0$ and $k>0$ set $g_t:= p_a(C_t) =1+ t(t+1)(2t-5)/6$, $g_{t,k}:= g_t+g_k = h^1(\Oo _{C_{t,k}})$ and $d_{t,k}:= \deg (C_{t,k}) =t(t+1)+k(k+1)$.

 \begin{remark}\label{oo1}
 Let $D_0\subset \PP^3$ be a smooth curve such that $h^1(\Oo _{D_0}(1)) =0$. Fix distinct lines $D_i$, $1\le i\le k$, such that $D_0\cup D_1\cup \cdots \cup D_k$
 is nodal and $1\le \sharp ((D_0\cup \cdots \cup D_{i-1})\cap D_i)\le 2$ for each $i=1,\dots ,k$. Then $D_0\cup \cdots \cup D_k$ is smoothable (\cite{hh}, \cite{s}).
 A Mayer-Vietoris exact sequence gives $h^1(\Oo _{D_0\cup \cdots \cup D_k}(1)) =0$ and hence $D_0\cup \cdots \cup D_k$ may be deformed to a non-special smooth curve.
 \end{remark}

We fix integers $m\ge 5$, $\alpha >0$ and take positive integers $t, k$ such that $t+k\equiv m\pmod{2}$ and $t\ge k$. For all integers $s\ge t+k-1$ with $s \equiv t+k-1 \pmod{2}$ we define
the integers $a(s,t,k)_\alpha$, $b(s,t,k)_\alpha$ and $g(s,t,k)_\alpha$ in the following way.
Set $a(t+k-1,t,k)_\alpha = b(t+k-1,t,k)_\alpha=g(t+k-1,t,k)_\alpha =0$. Define the integers $a(t+k+1,t,k)_\alpha$ and $b(t+k+1,t,k)_\alpha$
 by the relations
 \begin{align}\label{eqo4}
 &(t+k+1)(d_{t,k}+a(t+k+1,t,k)_\alpha) +3 -g_{t,k} \notag \\
 &+b(t+k+1,t,k)_\alpha =\binom{t+k+4}{3}, \ 0\le b(t+k+1,t,k)_\alpha \le t+k
 \end{align}
 
 Set $g(t+k+1,t,k)_\alpha:= 0$. Hence if $s\in \{t+k-1,t+k+1\}$ the integers $a(s,t,k)_\alpha$, $b(s,t,k)_\alpha$ and $g(s,t,k)_\alpha$ do not depend on $\alpha$. Fix an integer $s\ge t+k+3$ with $s\equiv t+k+1\pmod{2}$ and assume defined the integers $a(s-2,t,k)_\alpha$, $b(s-2,t,k)_\alpha$ and $g(s-2,t,k)_\alpha$.
 Define
 the integers $a(s,t,k)_\alpha$ and $b(s,t,k)_\alpha$ by the relations
 \begin{align}\label{eqo4.1}
 &2(d_{t,k} + a(s-2,t,k)_\alpha) + (s-1)(a(s,t,k)_\alpha -a(s-2,t,k)_\alpha) +\alpha+ \notag \\
 & +b(s,t,k)_\alpha -b(s-2,t,k)_\alpha =(s+1)^2, \ 0\le b(s,t,k) \le s-2 
 \end{align}
 Set $g(s,t,k)_\alpha = g(s-2,t,k)_\alpha + a(s,t,k)_\alpha -a(s-2,t,k)_\alpha-\alpha$. We claim that
 \begin{equation}\label{eqo5}
s(d_{t,k} +a(s,t,k)_\alpha)+3 -g_{t,k} -g(s,t,k)_\alpha 
+b(s,t,k)_\alpha =\binom{s+3}{3}
 \end{equation}
 To prove the claim use induction on $s$; add the equation in (\ref{eqo4.1}) to the case $s'= s-2$ of (\ref{eqo5}); start the inductive assumption with the case $s=t+k+1$ true by (\ref{eqo4}), (\ref{eqoov1}) and (\ref{eqoov0})). We have $b(t+k+1,t,k)_\alpha \le t+k$, $b(s,t,k)_\alpha \le s-2$ if $s\ge t+k+3$ and $g(s,t,k)_\alpha = a(s,t,k)_\alpha -a(t+k+1,t,k)_\alpha -\alpha(s-t-k-1)/2$ for all $s\ge t+k+3$. 
 
 Later (from Lemma \ref{ov7} on),
 we will assume $\alpha =202$ and write $a(s,t,k)$, $b(s,t,k)$ and $g(s,t,k)$ instead
 of $a(s,t,k)_{202}$, $b(s,t,k)_{202}$ and $g(s,t,k)_{\alpha}$.

 All our constructions work for any $\alpha \ge 105$, but we would get far worst bounds in some key numerical lemma
 taking $\alpha =105$. The culprits are the numerical lemmas, which give upper bounds for a certain integer $e$. We always need to take $\alpha \ge e+1$.
 To get $e\le 104$ (resp. $e\le 201$) in Lemmas \ref{oov2.1} and  \ref{oov4} (resp. Lemmas \ref{oov2.1=} and \ref{oov4=}) we need $t+k\ge 1113636$ and $s \ge 1157520$ (resp. $t+k \ge  42040$ and $s \ge 42674$). Taking $\alpha =202$ we get far better bounds for all results stated in the introduction (e.g. if we used $\alpha =105$ in Theorem \ref{i1} we would need roughly $m\ge 5\cdot 10^6$), but it is certainly not an optimal
 choice. It is not necessary to quote \cite{bef} to obtain our results (e.g. a very weak form of \cite{be2} would suffice), but the bounds would be far worse (roughly you square
 all lower bounds assumptions); here the culprit is the last part of the proof of Theorem \ref{i2}.
 
  \begin{remark}\label{ov8}
 Under suitable assumptions on $t$, $k$ and $\alpha$ we have $g(s,t,k)_\alpha\ge 26$ for
 all $s\ge t+k+3$ with $s\equiv t+k-1 \pmod{2}$ (Lemma \ref{oov6}). Hence a general curve $Y$ of genus $g(s,t,k)_\alpha$ and degree $a(s,t,k)_\alpha$ satisfies $h^1(N_Y(-2)) =0$ (\cite[page 67, inequality $DP (g) \le g + 3$]{pe}). We have $g(t+k+1,t,k)_\alpha =0$. A general smooth rational space curve $T$ of degree $x\ge 3$
has balanced normal bundle, i.e. $N_T$ is a direct sum of two line bundles of degree $2x-1$ (\cite[Proposition 6]{ev}), i.e. $h^1(N_T(-2)) =0$. Fix integers $q$ and $z \ge q+3$ such that a general non-special smooth curve $C\subset \PP^3$ of degree $z$ and genus $q$ satisfies $h^1(N_C(-2)) =0$ and
 hence (since $\chi (N_C(-2)) =0$) $h^0(N_C(-2)) =0$. Fix a smooth quadric $Q\subset \PP^3$ and take a general $B\subset Q$ with $\sharp (B)=2z$. By \cite[Theorem 5.12]{pe} there
 is a non-special smooth curve $Y\subset \PP^3$ of degree $z$ and genus $q$ such that $B = Y\cap Q$. The assumption ``$g(s,t,k)_\alpha\ge 26$ '' was the assumption made
 in \cite{bef}. After \cite{bef} was completed E. Larson proved a far better result (\cite[Theorem 1.4]{l}, \cite[Theorem 1.4]{v}). Using it we could get slightly weaker numerical assumptions in all results in \cite{bef} (and hence in the results of this paper), but without lowering the bounds in the assumption by an order of magnitude.
 \end{remark}

\begin{remark}\label{3i1}
Lemma \ref{u3.0} shows that to carry over all our steps we cannot have $t\gg k$, say we need $t\le 200k$. Lemma \ref{u3} shows that to carry over the last few steps we also need
$t\ge 30k$.
\end{remark}

\section{Assertions $A(s,t,k)_\alpha$ and $A(s,t,k)$}\label{Sa}

 As in section \ref{Spr} we fix an integer $m\ge 5$, a positive integer $\alpha$ and take positive integers $t, k$ such that $t+k\equiv m\pmod{2}$.  
 
For any integer $s\ge t+k+1$ with $s\equiv t+k+1 \pmod{2}$ we define the following Assertion $A(s,t,k)_\alpha$ (we use the lemmas in  the next section  to make sense of it). We will prove $A(s,t,k)_\alpha$ for the quadruples $(\alpha,s,t,k)$ we need in section \ref{SprA}.

 Let $Q$ be a smooth quadric and $W\subset \PP^3$ a reduced curve such that $Q\cap W$ is formed by $2\deg (W)$ points and no line of $Q$ contains two or more points of $W$.
Let $E\subset Q$ be a finite set. Take positive integers $a$, $b$. An \emph{$(a,b)$-grid of $Q$ adapted to $(W,E)$} is a union $T\subset Q$ of $a+b$ distinct lines of $Q$, $a$ of them of bidegree $(1,0)$, $b$ of them of bidegree $(0,1)$, each line of $T$ contains a point of $W\cap Q$, no two lines of $T$ contain the same point  of $W\cap Q$ and $T\cap E=\emptyset$. Obviously the existence on an $(a,b)$-grid for $(W,E)$ implies $2\deg (W) \ge a+b$. We have $\mathrm{Sing}(T) =(a-1)(b-1)$. Since no two lines of $T$ contain the same point of $W\cap Q$, we have $\mathrm{Sing}(T) \cap W =\emptyset$.

\quad {\bf {Assertion}} $A(s,t,k)_\alpha$, $s\ge t+k+1$ with $s\equiv t+k+1 \pmod{2}$: Set $\delta := a(s+2,t,k)_\alpha
-a(s,t,k)_\alpha$. Let $e$ be the maximal positive integer such that $b(s,t,k)_\alpha > (e-1)(\delta-e-1)$ and $e\le \delta
/2$. Let $Q$ be a smooth quadric. Fix $C_{t,k}$ intersecting transversally $Q$ and such that $Q\cap C_{t,k}$ is formed by
$2d_{t,k}$ general
 points of $Q$. We call $A(s,t,k)_\alpha$ the existence of a triple $(Y,T,S)$ with the following properties:
 \begin{enumerate}
 \item $Y$ is a smooth and connected curve of degree $a(s,t,k)_\alpha$ and genus $g(s,t,k)_\alpha$ such
 that $Y\cap C_{t,k}=\emptyset$, $Y$ intersects transversally $Q$ and $(C_{t,k}\cup Y) \cap Q$ is general in $Q$. 
 \item $T$ is a grid of type $(e,\delta -e)$ for $(Y,E)$, where $E:= C_{t,k}\cap Q$. $S\subset \mathrm{Sing}(T)$.
 \item $h^i(\Ii _{Y\cup C_{t,k}\cup S}(s)) =0$, $i=0,1$. 
 \end{enumerate}
 
 In the next section we collect the numerical lemmas needed to prove $A(s,t,k)_\alpha$ for certain quadruples $(\alpha,s,t,k)$.
 
\medskip 
The proof is by induction on $s$. We may give a rough outline with the following pictures.

\begin{center}
\scalebox{.7}{\includegraphics{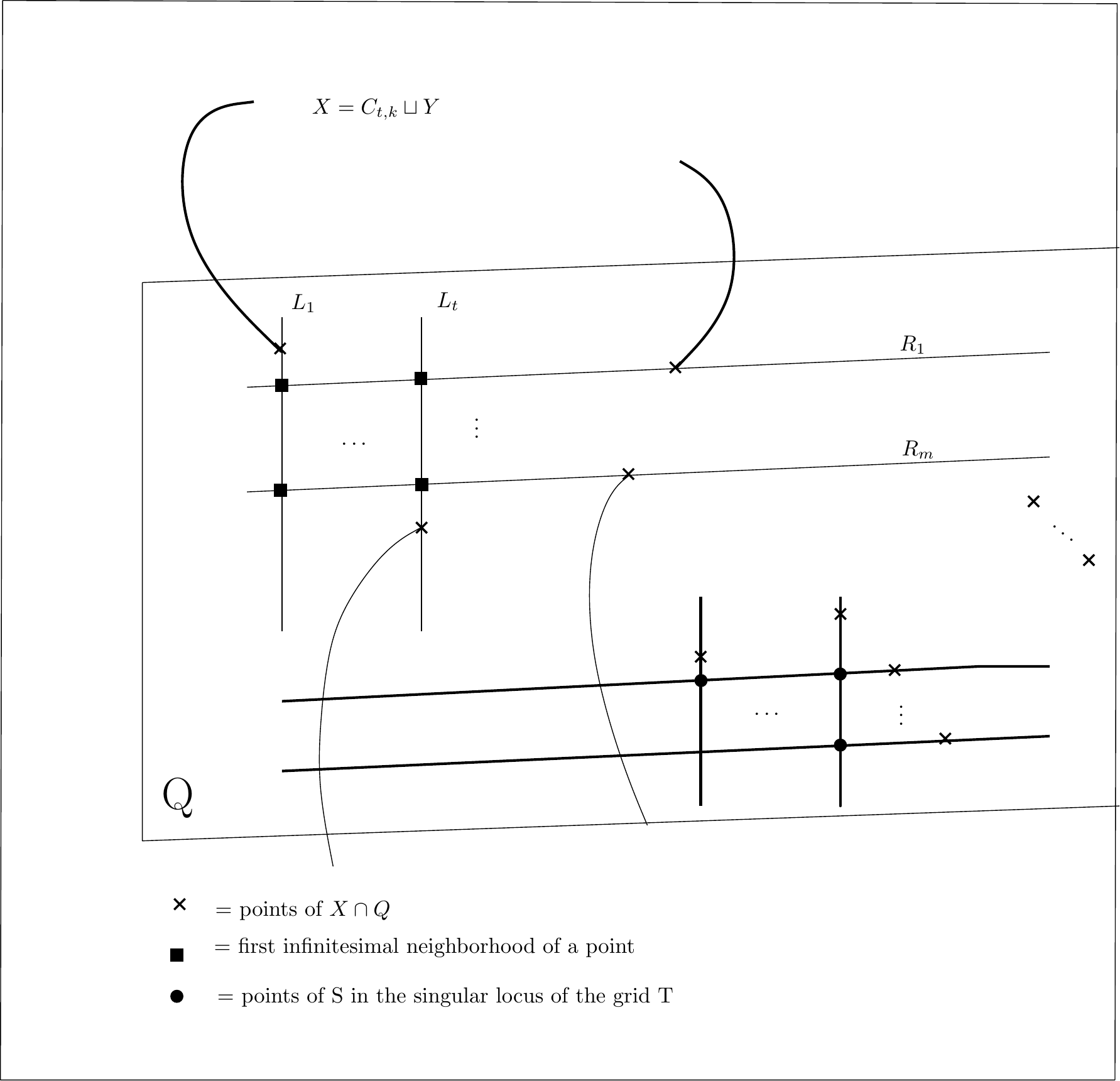}}
\end{center}
We consider $Z = X \cup (\cup _iL_i) \cup (\cup_jR_j)\cup \chi \cup S$. Here $\chi$ denotes the set of nilpotents (pictured with a black square). Assume $F$ is a form of degree $s+2$ vanishing on $Z$. We look at $F|Q$: it vanishes on $D = (\cup _iL_i) \cup (\cup_jR_j)$, on the points of $X\cap (Q\setminus D)$ and on the points of $S$. We check that this implies $F|Q=0$. It follows that $F=QG$, where $G$ is a form of degree $s$ vanishing on the residual scheme $Res_Q(Z)$. This residual scheme is:
\begin{center}
\scalebox{.7}{\includegraphics{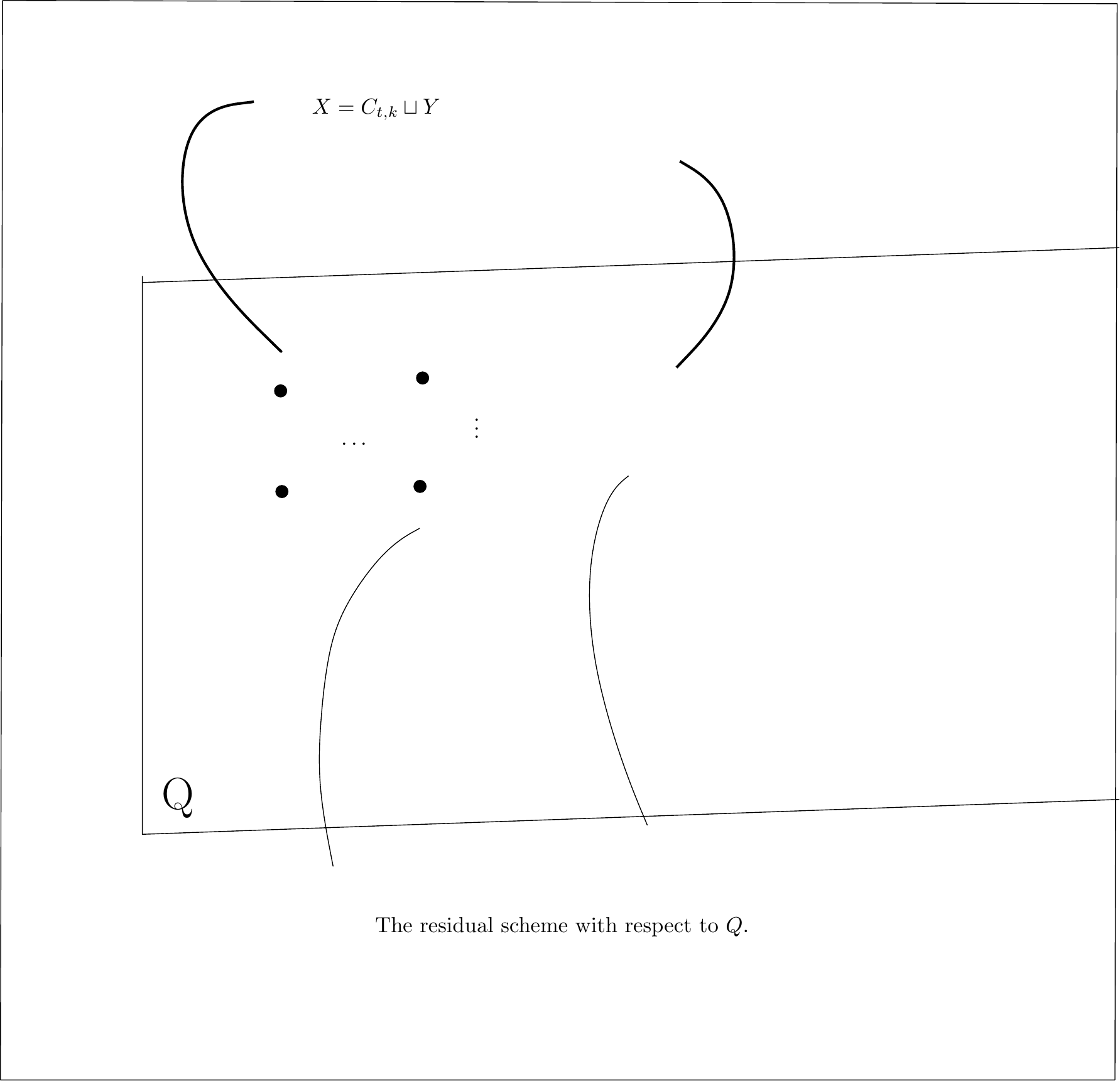}}
\end{center}
Observe that the set of points $\chi _{red}$ lies on a grid (build by the lines $L_i, R_j$) on $Q$. Since we have taken $X \cup \chi _{red}$ satisfying $A(s,t,k)$, we get $G=0$. Hence $F=0$. Then we show that $X \cup (\cup _iL_i) \cup (\cup_jR_j)\cup \chi$ is smoothable.  By semi-continuity this shows $A(s,t,k) \Rightarrow A(s+2,t,k)$.  

\medskip

We explain here the numerical restriction needed just to know that $A(s,t,k)_\alpha$ is well-defined (e.g. we need $\delta \ge
0$ and if
$b(s,t,k)_\alpha > 0$ we also need $e\le \delta/2$.
 
 \begin{enumerate}
 \item For $s\ge t+k+3$ we need the existence of a non-special curve $Y$ with degree $d(s,t,k)_\alpha$ and genus
$g(s,t,k)_\alpha$ such that $Y\cap Q$ is formed by $2a(s,t,k)_\alpha$ general points of $Q$. To get this property it is
sufficient to use that
$g(s,t,k)_\alpha \ge 26$ (Remark \ref{ov8}) and quote \cite[page 67, inequality $DP (g) \le g + 3$]{pe}). Since $(Y\cup
C_{t,k}\cap Q$
 is general in $Q$, no two points of it are contained in the same line of $Q$. 
 \item We need $\delta \ge 0$ and this is true by Lemma \ref{oov2}.
  \item When $b(s,t,k)_\alpha>0$ we need $e\le \delta/2$. For this we use Lemmas \ref{oov2},  \ref{oov2.1} and \ref{oov2.1=}.
  \item We need $2a(s,t,k)_\alpha \ge \delta$ to get a degree $\delta$ grid $T\subset Q$ such that each irreducible component
of $T$ contains a point of $T\cap Q$.
 \end{enumerate}
 
 Now we explain with examples why we need lower bounds on $k$ to have a chance that $A(s,t,k)_\alpha$ is true when $s-t-k$ is small. In Lemma
\ref{u3.0} we will show that it is important that $t$ is not too large with respect to $k$.
 
 \begin{example}\label{exk1}
 First take $k=1$. We have $d_{t,1} =d_t+2$ and $g_{t,1} =g_t$. Since $td_{t,1}+2-g_{t,1} =\binom{t+3}{3}$ and $\binom{t+5}{3}-\binom{t+3}{3} =(t+1)^2$,
(\ref{eqo4}) gives
 $2d_{t,1} +(t+2)a(t+2,t,1)_\alpha +b(t+2,t,1)_\alpha) = (t+3)^2$ and hence $a(t+2,t,1)_\alpha\le 5$, which is not enough to
attach any grid  to $Y\cap Q$ when $\deg (Y) =a(t+2,t,1)_\alpha$.
 
 Now take $k$ arbitrary. We show that if $t\gg k$ we cannot find a grid $T\subset Q$ with $\deg (T) =\delta$ and with $\sharp (\mathrm{Sing}(T)) \ge s-2$ and thus if $b(s,t,k)$ is large there is no $S\subseteq \mathrm{Sing}(T)$ with $\sharp (S) =b(s,t,k)$. The maximum integer $\sharp (\mathrm{Sing}(T))$ for grids of bidegree $(e,\delta -e)$
 is $\lfloor \delta/2\rfloor \lceil \delta/2\rceil$.
 We take $s=t+k+3$ and any fixed $\alpha$. For $t\gg k$ we get $a(t+k+1,t,k)_\alpha\sim 4k$ and $\delta:=
a(t+k+3,t,k)_\alpha-a(t+k+1,t,k)_\alpha
\sim 8k \ll s$.  There are other numerical problems if $t\gg k$, most of them arising in the part  in which $g$ is used (Lemma
\ref{u3.0}).
 \end{example}

 \section{Numerical lemmas, I}\label{SnI}

For any integer $s\ge t+k-1$ set $I_{t,k}(s):= h^0(\Ii _{C_{t,k}}(s))$. Since $\deg (C_{t,k})=d_{t,k}$, $h^1(\Oo _{C_{t,k}})=g_{t,k}$ and $h^1(\Ii _{C_{t,k}}(s)) =0$ for $s\ge t+k-1$, we have
\begin{equation}\label{eqoov1}
I_{t,k}(s) = \frac{(s-t-k+1)}{6}[(s-t-k+3)(s-t-k+2)+3(t+k)(s-t-k+2)+6kt]
\end{equation}
Note that $I_{t,k}(t+k-1)=0$. By (\ref{eqo4}) we have
\begin{equation}\label{eqoov0}
(s-1)a(s,t,k)_\alpha +b(s,t,k)_\alpha = I_{t,k}(s) -4 -\alpha (s-t-k-1)/2
\end{equation}
\begin{lemma}\label{oov1}
We have
$$I_{t,k}(s+2)-I_{t,k}(s) =(s+3)^2-(t^2+k^2+t+k)$$
and if $u\in \NN$ and $s =t+k+1+2u$ then
$$I_{t,k}(s) = \frac{4}{3}u^3+(6+2(t+k))u^2 +(\frac{26}{3} +5(t+k)+2tk)u +4 + 3(t+k)+2kt$$
\end{lemma}
\begin{proof}
Let $Q$ be a general quadric surface and consider the exact sequence:\\ $0 \to \Ii _{C}(s) \to \Ii _C(s+2) \to \Ii _{C\cap Q}(s+2) \to 0$ (where $C := C_{t,k}$). Since $h^i(\Ii _C(m))=0$, if $m \geq t+k-1, 1\leq i \leq 2$, we get: $I(s+2) - I(s) = h^0(\Ii _{C\cap Q}(s+2)) = h^0(\Oo _Q(s+2)) - 2d_{t,k} = (s+3)^2 - (t^2+t+k^2+k)$.

For the second assertion one can make a direct computation or proceed by induction using: $0 \to \Ii _{C}(s) \to \Ii _C(s+1) \to \Ii _{C\cap H}(s+1) \to 0$, where $H$ is a general plane, and taking into account that $h^0(\Ii _{C\cap H}(s+1)) = h^0(\Oo _H(s+1))-d_{t,k}$. 
\end{proof}
We fix an integer $\alpha \ge 0$. Write $s = t+k+1+2u$ with $u\in \NN$. By (\ref{eqoov0}) we have
\begin{align}&a(s,t,k)_\alpha(2u+t+k) + b(s,t,k)_\alpha = \notag\\
&\frac{4u^3}{3} + [6+2(t+k)]u^2 +[\frac{26}{3}+5(t+k) +2kt]u +4+3(t+k)+2kt -\alpha u \notag\end{align}
By (\ref{eqo4.1}) we have
\begin{equation}\label{eqoov2}
(s-1)a(s,t,k)_\alpha + b(s,t,k)_\alpha = I_{t,k}(s) -4 -\alpha (s-t-k-1)/2
\end{equation}
We compare $a(s,t,k)_\alpha$, $s=t+k+1+2u$, with the function
$$\psi (u):= 
\frac{2u^2}{3} + [\frac{2(t+k)}{3}+3]u + t+k$$
\begin{lemma}\label{oov3}
If $t \ge 4$ and $k\ge 4$, then $a(s,t,k)_\alpha \le \psi (u)$, i.e.
$$a(s,t,k)_\alpha \le \frac{s^2+7s}{6} -\frac{(t+k)^2}{6} -\frac{3(t+k)}{6} -\frac{8}{6}$$
\end{lemma}
\begin{proof}
Note that
$$(2u+t+k)\psi (u) =\frac{4u^3}{3} +[2(t+k)+6]u^2 + [\frac{2(t+k)^2}{3}+5(t+k)]u +(t+k)^2$$
Since $b(s,t,k)_\alpha \ge 0$ and $\alpha u\ge 0$, we get
\begin{equation}\label{eqoov3}
(2u+t+k)(\psi (u) -a(s,t,k)_\alpha \ge [\frac{2(t+k)^2}{3} -2kt -\frac{26}{3}]u +(t+k)^2-3(t+k)-2kt
\end{equation}
We have $\frac{2(t+k)^2}{3} -2kt -\frac{26}{3}\ge 0$ $\Leftrightarrow$  $2(t+k)^2\ge 6kt+26$ $\Leftrightarrow$   $(t-k)^2+t^2+k^2\ge 26$, which is satisfied
if $t\ge 4$ and $k\ge 4$. We have $(t+k)^2 -3(t+k)-2kt \ge 0$  $\Leftrightarrow$  $t^2+k^2-3(t+k) \ge 0$  $\Leftrightarrow$  $t(t-3)+k(k-3)\ge 0$,
which is true if $k\ge 3$ and $t\ge 3$. Thus the right hand side of (\ref{eqoov3}) is non-negative. For the last assertion we use that $\psi (u)
= \frac{(s-t-k-1)}{6}[(s-k-t-1)+2(t+k)+9] +(t+k) = \frac{(s-t-k-1)(s+t+k+8)}{6}+t+k$.
\end{proof}

\begin{lemma}\label{oov2}
Fix integers $\alpha \ge 0$, $u\ge 0$, and set $s:= t+k+1+2u$ and $\delta := a(s+2,t,k)_\alpha -a(s,t,k)_\alpha$. Assume $t\ge k \ge
t/200$ (resp. $t\ge k \ge t/3$ and $k\ge 4$). Then
$$\delta > \frac{s}{102} -\frac{\alpha}{s+1}$$
 (resp. $\delta > \frac{s}{3}-\frac{\alpha}{s+1})$. In particular, if $s+1 > \alpha$, then $\delta > -1 +s/102$ (resp. $\delta > -1 +s/3$).
\end{lemma}

\begin{proof}
Taking the difference of (\ref{eqoov2}) and the same equation for the integer $s'=s+2$, we get
\begin{equation}\label{eqoov4}
(s+1)\delta + 2a(s,t,k)_\alpha +b(s+2,t,k)_\alpha -b(s,t,k)_\alpha = I_{t,k}(s+2) -I_{t,k}(s) -\alpha
\end{equation}
By Lemma \ref{oov1} we have
$$(s+1)\delta +2a(s,t,k)_\alpha + b(s+2,t,k)_\alpha -b(s,t,k)_\alpha =(s+3)^2-(t^2+k^2+t+k) -\alpha$$
We have $b(s,t,k)_\alpha \ge 0$, $b(s+2,t,k)_\alpha \le s$ and $2a(s,t,k)_\alpha \le [(s^2+7s)-(t+k)^2-3(t+k)-8]/3$ (Lemma \ref{oov3}). Thus
$$ (s+1)\delta \ge (s+3)^2 -(t^2+k^2+t+k) -2(s^2+7s)/3 + 2(t+k)^2/3 +2(t+k) +16/3 -\alpha$$, i.e.
$$(s+1)\delta \ge \frac{(s^2+s)}{3} -\frac{(t^2+k^2-4tk -3t-3k-1)}{3} -\alpha.$$ 
We obviously have $3t+3k+1 \ge 0$. If $t\le 3k$, then $(t^2+k^2-4tk -3t-3k-1) <0$ (more precisely, it is sufficient to
assume $t < \frac{4k+3+\sqrt{12k^2+36k+73}}{2}$) and hence we conclude in this case. Now assume only $t\le 200k$. Since
$s\ge t+k$, it is sufficient to prove that $t^2+k^2-4tk \le \frac{33}{34}(t+k)^2$, i.e.
$t^2+k^2 -202tk\le 0$. This is true if $t\ge k>0$ and $t/k \le 101 +\sqrt{101^2-1}$.
\end{proof}

Note that $a(t+k+1,t,k)_\alpha$ and $b(t+k+1,t,k)_\alpha$ do not depend from $\alpha$.

\begin{lemma}
\label{L-recI}
Let $A$ and $B$ be positive rational numbers. Assume $t+k \geq 4A + 2B$ and $(s-1)d_{t,k} + I_{t,k}(s) \geq As^2 + Bs$, for some $s \geq t+k+1$. Then $(s+1)d_{t,k} + I_{t,k}(s+2) \geq A(s+2)^2 + B(s+2)$.
\end{lemma}

\begin{proof} By Lemma \ref{oov1}, we have $I(s+2) - I(s) = (s+3)^2 -(t^2+k^2 +t +k)$ (where $I(m):= I_{t,k}(m)$). It follows that:
\begin{equation}
\begin{split}
(s+1)d_{t,k} + I(s+2) = (s-1)d_{t,k} + I(s) + 2d_{t,k} + (s+3)^2 -(t^2 +k^2 +t+k)\\ 
\geq As^2 +Bs +  2d_{t,k} + (s+3)^2 -(t^2 +k^2 +t+k)
\end{split}
\end{equation}
So it is enough to prove:
\begin{equation}
 2d_{t,k} + (s+3)^2 -(t^2 +k^2 +t+k) = (s+3)^2 \geq 4As + 4A +2B
 \end{equation}
 But $s^2 + 6s \geq 4As +4A +2B \Leftrightarrow s+6 \geq 4A +(4A +2B)/s$ and this last inequality follows from: $s +6 > s \geq 4A + (4A+2B)/s$, since $s \geq t+k+1 > 4A +2B$. 

\end{proof}

\begin{lemma}
\label{ineq(a(s))d} 
Let $C > 0$ and $D$ be rational numbers. Assume $t \geq k \geq 2C$ and $t+k \geq 4C + 2|D-C+\alpha +2|$. Then we have $a(s,t,k) + d_{t,k} \geq Cs + D$ for any $s \geq t+k+1$, $s \equiv t+k+1 \pmod{2}$.
\end{lemma}

\begin{proof} By definition $a(s)(s-1) + b(s) = I(s) -4-\alpha (s-t-k-1)/2$, with $0 \leq b(s) \leq s-2$ (we drop the indices $t,k$ in $a, b, I$). Using $b(s) \leq s-2$, we get:
\begin{equation}
\begin{split}
a(s)(s-1) \geq I(s) -2 -\alpha (s-t-k-1)/2 -s\\
\geq I(s) -s(\alpha +2)/2
\end{split}
\end{equation}
By dividing by $s-1$ and using $s(\alpha +2)/2(s-1) \leq \alpha +2$, we get: 
$$a(s) \geq I(s)/(s-1) -(\alpha +2)$$
So it is enough to show: $d_{t,k} + I(s)/(s-1) \geq Cs + D +\alpha +2$, which is equivalent to:
$$(s-1)d_{t,k} + I(s) \geq As^2 +Bs \,\,\,\,(*)$$
where $A = C$ and where $B = D-C+\alpha +2$. Note that $B$ can be negative so we will consider the following inequality:
$$(s-1)d_{t,k} + I(s) \geq As^2 +|B|s \,\,\,\,(**).$$
Clearly $(**)$ implies $(*)$.
According to Lemma \ref{L-recI}, if $t+k \geq 4A +2|B|$, it is enough to prove this inequality for $s = t+k+1$. Since $I(t+k+1) = 2tk + 3(t+k)+4$, for $s=t+k+1$ $(**)$ reads like:
\begin{equation}
\begin{split}
t^2(t+1)/2 + k^2(k+1)/2 + tk(t+k+2)/2 + 2tk +3(t+k)+4 \\ \geq  At^2 +Ak^2 + 2Atk +(t+k+1)(2A+|B|) -A
\end{split}
\end{equation}
We have $t^2(t/2) \geq At^2$, $k^2(k/2) \geq Ak^2$ (because $t \geq k \geq 2A$). For the same reason, we also have $tk(t+k+2)/2 \geq tk(2A+1)$. It remains to show that $(t^2 + k^2)/2 + 3tk + 3(t+k+1)+1 \geq (t+k+1)(2A+|B|) -A$. We have:
$$(t^2+k^2)/2 + 3tk +3(t+k+1) +1 = (t+k+1)^2/2 + 2tk + 2t +2k +7/2$$
Since $(t+k+1)/2 \geq 2A + |B| + 1/2$ by assumption, we get inequality $(**)$, hence also inequality $(*)$.  
\end{proof}

\begin{corollary}
\label{C-boundDelta}
Let $M$ be a positive rational number. Assume $t \geq k \geq M +6$ and $t+k \geq 2M +\alpha +18$, then $\delta := a(s+2) - a(s) < s -M$, $\forall s \geq t+k+1$, $s \equiv t+k+1 \pmod{2}$.
\end{corollary}

\begin{proof} We have: $2(d_{t,k} + a(s)) + (s+1)\delta + \alpha + b(s+2) -b(s) = (s+3)^2$ by (\ref{eqo4.1}). If $\delta \geq s-M$, since $b(s+2) \geq 0$ and $b(s) \leq s-2$, we get: $(s+3)^2 \geq 2(d_{t,k} + a(s)) + (s+1)(s-M) + \alpha -s+2$, i.e. $s(M+6) + M-\alpha +7 \geq 2(d_{t,k} + a(s))$. Using Lemma \ref{ineq(a(s))d} with $C = M/2 +3$ and $D = (M-\alpha )/2 +4$, we get a contradiction. 
\end{proof}


\begin{lemma}\label{oov2.1}
Assume $t\ge k\ge 4$ and $t\le 200k$ (resp. $t\le 3k$). Then $a(t+k+1,t,k)_\alpha \ge (t+k)/102$ (resp. $a(t+k+1,t,k)_\alpha
\ge (t+k)/3$). If $\alpha \le -1+(t+k)/102$ (resp. $\alpha \le -1+(t+k)/3$), then $a(t+k+3,t,k)_\alpha -a(t+k+1,t,k)_\alpha
\le 1+2a(t+k+1,t,k)_\alpha$. Let $e$ be the minimal integer $x$ such that $1\le x \le \delta:= a(t+k+3,t,k)_\alpha
-a(t+k+1,t,k)_\alpha$ and
$b(t+k+1,t,k)_\alpha
\le (x-1)(\delta -x-1)$. Assume $\alpha \le -1+(t+k)/102$ and $t+k\ge 1113636$ (resp. $\alpha \le -1+(t+k)/3$ and $t+k>78$). We have $e\le 104$ (resp. $e\le 5$).
\end{lemma}

\begin{proof}
Since
$(t+k+2)^2 -(t^2+k^2+t+k) = 2tk +3t+3k+4$, Lemma
\ref{oov1} for
$s=t+k-1$ gives
$(t+k+1)a(t+k+1,t,k)_\alpha+1+b(t+k+1,t,k)_\alpha = 2tk+3t+3k+4$.
Since $b(t+k+1,t,k)_\alpha \le t+k$, we get
\begin{equation}\label{eq==o1}
(t+k+1)a(t+k+1,t,k)_\alpha \ge 2tk+2t+2k+3
\end{equation} If $a(t+k+1,t,k)_\alpha \le (t+k-1)/102$ (resp. $a(t+k+1,t,k)_\alpha  \le
(t+k-1)/3$, then $(t+k+1)(t+k-1) \ge 102(2tk+2t+2k+3)$
(resp. $(t+k+1)(t+k-1) \ge 6tk+12t+12k+18$), which is false for all positive $k$ if $t\le 200k$ (resp. $t\le 3k$).

Now assume $\alpha \le -1+(t+k)/102$ (resp. $\alpha \le -1+(t+k)/3$). To prove the second assertion of the lemma it is sufficient to prove
that $a(k+t+3,t,k)_\alpha +1 \le 3a(k+t+1,t,k)_\alpha$. We have $I_{t,k}(t+k+3) =
\frac{4}{6}[(30+15(t+k)+6kt] = 20+10(t+k)+4kt$ and $a(t+k+3,t,k)_\alpha \le (I_{t,k}(t+k+3)+4)/(t+k+2)
=(20+10(t+k)+4kt)/(t+k+2)$, while $a(t+k+1)_\alpha \le (I_{t,k}(t+k+1) -1-t-k)/(t+k+1) = 2(t+k+tk)/(t+k+1)$.

Since
$b(t+k+1,t,k)_\alpha
\le t+k$, to prove the last assertion of the lemma it is sufficient to observe that $103(\delta -202)
\ge t+k$ (resp. $4\delta \ge
24+t+k$) by the case $s=t+k+1$ of Lemma
\ref{oov2} and the assumption $t+k\ge 1113636=106\cdot 103\cdot 102$, which gives that $103(-1+\frac{t+k+1}{102} -202) \ge
t+k$.\end{proof}

\begin{lemma}\label{oov2.1=}
Assume $200k\ge t\ge k\ge 4$, $\alpha \le -1+(t+k)/102$ and $t+k \ge  42040$.  Let $e$ be the minimal integer $x$ such that
$1\le x
\le \delta:= a(t+k+3,t,k)_\alpha -a(t+k+1,t,k)_\alpha$ and
$b(t+k+1,t,k)_\alpha
\le (x-1)(\delta -x-1)$. Then $e\le 201$.
\end{lemma}

\begin{proof}
Assume $e\ge 202$. Since $b(t+k+1,t,k)_\alpha \le t+k$, we get $201(\delta -203) \\e k+t-1$. Since $\delta \le -1 + (t+k)/102$
by Lemma \ref{oov2.1}, we get $99(k+t) \le 203\cdot 201\cdot 102 -102$, which is false if $k+t\ge 42040$.
\end{proof}

\begin{lemma}\label{oov4}
Fix integers $\alpha \ge 0$, $u\ge 0$, and set $s:= t+k+1+2u$ and $\delta := a(s+2,t,k)_\alpha -a(s,t,k)_\alpha$. Assume $s+1 >\alpha$. Assume $t\ge
k \ge t/200$ and $k\ge 4$  and $s \ge 1157520$ (resp. $s >78$, $t\ge k \ge t/3$ and $k\ge 4$). Let $e$ be the minimal integer $x$ such that $1\le x \le
\delta$ and
$b(s,t,k)_\alpha \le (x-1)(\delta -x-1)$. Then $e\le 104$ (resp. $e\le 5)$.
\end{lemma}

\begin{proof}
The integer $e$ is defined by Lemma \ref{oov2}. Since $b(s,t,k)_\alpha \le s-2$, it is enough to check that $s-2 \le
103(\delta -202)$ (resp. $s-2
\le 4(\delta -6)$, which is true, because Lemma \ref{oov2} gives $\delta \ge -1 +\frac{s}{102}$ (resp. $\delta >-1 +\frac{s}{3}$)
and we assumed $s \ge (2+103\cdot 106)\cdot 106 =1157520$ (resp. $s>78$).
\end{proof}

\begin{lemma}\label{oov4=}
Fix integers $\alpha \ge 0$, $u\ge 0$, and set $s:= t+k+1+2u$ and $\delta := a(s+2,t,k)_\alpha -a(s,t,k)_\alpha$. Assume $s+1 >\alpha$ and $t\ge
k \ge t/200$,  $k\ge 4$  and $s \ge 42674$. Let $e$ be the minimal integer $x$
such that $1\le x \le
\delta$ and
$b(s,t,k)_\alpha \le (x-1)(\delta -x-1)$. Then $e\le 201$.
\end{lemma}

\begin{proof}
The integer $e$ is defined by Lemma \ref{oov2}. Since $b(s,t,k)_\alpha \le s-2$, it is enough to check that $s-2 \le
98(\delta -203)$, which is true, because Lemma \ref{oov2} gives $\delta \ge -1 +\frac{s}{102}$ and we assumed $s \ge 42674$.
\end{proof}

\begin{lemma}\label{==o1}
If $t\ge k$ and $k^2\ge (\alpha +5)(k+2)$, then $2tk \ge (\alpha +5)(t+k+4)$.
\end{lemma}

\begin{proof}
Set $\gamma (t,k):= 2tk-(\alpha +5)(t+k+4)$. We have $\gamma (k,k) \ge 0$ if $k^2\ge (\alpha +5)(k+2)$. To get the lemma use that $\partial _t \gamma (t,k)
= 2k -\alpha -5$.
\end{proof}

\begin{lemma}\label{oov5}
Assume $ k\le t \le 200k$, $2tk \ge (\alpha +5)(t+k+4)$, $\alpha \le -1+(t+k)/102$ and take an integer $s\ge t+k+1$ with 
$s\equiv t+k+1 \pmod{2}$. Then $\delta := a(s+2,t,k)_\alpha-a(s,t,k)_\alpha \le 2a(s,t,k)_\alpha
-\alpha$ and $\tau := a(s+4,t,k)_\alpha -a(s+2)_\alpha \leq
2a(s,t,k)_\alpha +
\alpha -1$.
\end{lemma}

\begin{proof}
Assume $\delta \ge 2a(s,t,k)_\alpha -\alpha +2$. Since $b(s+2,t,k)_\alpha \ge 0$ and $b(s,t,k)_\alpha \le s-1$, (\ref{eqoov4}) gives
\begin{equation}\label{eqoov5} 2d_{t,k} + (2s+4)a(s,t,k)_\alpha  \le (s+3)^2+\alpha (s+1) -s-3\end{equation}
First assume $s=t+k+1$. In this case (\ref{eqoov5}) is equivalent to
\begin{equation}\label{eqoov5.1}
(2t+2k+6)a(t+k+1,t,k)_\alpha \le \alpha (t+k+2) +2tk +6t+6k+12
\end{equation}
From (\ref{eq==o1}) we get $4a(t+k+1,t,k)_\alpha +2tk\le \alpha (t+k+2) +2t+2k+6$, which is false if $2tk \ge (\alpha +2)(t+k+2)$.

Now assume $s\ge t+k+3$ and that (\ref{eqoov5}) is false for the integer $s-2$. Since $(s+3)^2+\alpha (s+1) -(s+1)^2 -\alpha (s-1) = 4s+2\alpha +8$, it
is sufficient to use that $(2s+4)a(s,t,k)_\alpha -2s\cdot a(s-2,t,k)_\alpha = 4a(s-2,t,k)_\alpha +(2s+4)(a(s,t,k)_\alpha-a(s-2,t,k)_\alpha) \ge 4(k+3) +4s$ by Lemmas \ref{oov2.1} and \ref{oov2}.

Now assume $\tau \ge  2a(s,t,k)_\alpha +2-\alpha$. From (\ref{eqo4.1}) for the integer $s'= s+4$ and using that $b(s+4,t,k)_\alpha \le s+2$
and $b(s+2,t,k)_\alpha \ge 0$ and $2a(s+2,t,k) \ge 2a(s,t,k) -2+s/51$, we get
\begin{equation}\label{eqoov4}
2d_{t,k} -2+s/51+ (2s+8)a(s,t,k)_\alpha +2s+6 -(s+3)\alpha \le (s+5)^2
\end{equation}
First assume $s=t+k+1$. In this case (\ref{eqoov4}) is equivalent to
\begin{equation}\label{eqoov4.1}
(2t+2k+10)a(t+k+1,t,k)_\alpha \le \alpha (t+k+4) +2tk +9t+9k+30 -(t+k+1)/51
\end{equation}
From (\ref{eqoov4.1}) we get $8a(t+k+1,t,k)_\alpha + 2tk \le \alpha (t+k+4) +5t+5k+24-(t+k+1)/51$, which is false if $2tk \ge (\alpha +5)(t+k+4)$.

Now assume $s\ge t+k+3$ and that (\ref{eqoov4}) is false for the integer $s-2$. Since $(s+5)^2 -(s+3)^2 = 4s+16$ and $2s+8-2(s-2)-8=4$, to get that  (\ref{eqoov4}) is false for the integer $s$
it is sufficient to use that $a(s,t,k)_\alpha -a(s-2,t,k)_\alpha \ge 2+\alpha$, which is true by Lemma \ref{oov2} and the assumption $s\ge t+k+3$, our assumptions on $t$, $k$ and $\alpha$.
\end{proof}

\begin{lemma}\label{oov6}
Assume $t \ge k\ge t/200$, $t+k\ge 102 (\alpha +27)$. Then
$g(s,t,k)_\alpha
\ge 26$.
\end{lemma}

\begin{proof}
We first do the case $s=t+k+3$. We have $g(t+k+3,t,k)_\alpha = a(t+k+3,t,k)_\alpha -a(t+k+1,t,k)_\alpha -\alpha$. Lemma
\ref{oov2} gives $a(t+k+3,t,k)_\alpha -a(t+k+1,t,k)_\alpha > -1 +(t+k+1)/102$ (resp. $a(t+k+3,t,k)_\alpha
-a(t+k+1,t,k)_\alpha > -1 +(t+k+1)/3$). Now assume $s\ge t+k+5$. By induction on
$s$ it is sufficient to prove that $g(s,t,k)_\alpha  \ge g(s-2,t,k)_\alpha$,
i.e. $a(s,t,k)_\alpha -a(s-2,t,k)_\alpha \ge
\alpha$, which is true by Lemma
\ref{oov2}.
\end{proof}

\begin{lemma}\label{oov2.1}
Assume $t\ge k\ge 4$ and $t\le 200k$. Then $a(t+k+1,t,k)_\alpha \ge (t+k)/102$. If $\alpha \le -1+(t+k)/102$, then $a(t+k+3,t,k)_\alpha -a(t+k+1,t,k)_\alpha
\le 1+2a(t+k+1,t,k)_\alpha$. Let $e$ be the minimal integer $x$ such that $1\le x \le \delta:= a(t+k+3,t,k)_\alpha
-a(t+k+1,t,k)_\alpha$ and
$b(t+k+1,t,k)_\alpha
\le (x-1)(\delta -x-1)$. Assume $\alpha \le -1+(t+k)/102$ and $t+k\ge 1113636$. We have $e\le 104$.
\end{lemma}

\begin{proof}
Since
$(t+k+2)^2 -(t^2+k^2+t+k) = 2tk +3t+3k+4$, Lemma
\ref{oov1} for
$s=t+k-1$ gives
$(t+k+1)a(t+k+1,t,k)_\alpha+1+b(t+k+1,t,k)_\alpha = 2tk+3t+3k+4$.
Since $b(t+k+1,t,k)_\alpha \le t+k$, we get
$(t+k+1)a(t+k+1,t,k)_\alpha \ge 2tk+2t+2k+3$. If $a(t+k+1,t,k)_\alpha \le (t+k-1)/102$, then $(t+k+1)(t+k-1) \ge 102(2tk+2t+2k+3)$, which is false for all positive $k$ if $t\le 200k$.

Now assume $\alpha \le -1+(t+k)/102$. To prove the second assertion of the lemma it is sufficient to prove
that $a(k+t+3,t,k)_\alpha +1 \le 3a(k+t+1,t,k)_\alpha$. We have $I_{t,k}(t+k+3) =
\frac{4}{6}[(30+15(t+k)+6kt] = 20+10(t+k)+4kt$ and $a(t+k+3,t,k)_\alpha \le (I_{t,k}(t+k+3)+4)/(t+k+2)
=(20+10(t+k)+4kt)/(t+k+2)$, while $a(t+k+1)_\alpha \le (I_{t,k}(t+k+1) -1-t-k)/(t+k+1) = 2(t+k+tk)/(t+k+1)$.

Since
$b(t+k+1,t,k)_\alpha
\le t+k$, to prove the last assertion of the lemma it is sufficient to observe that $103(\delta -202)
\ge t+k$ by the case $s=t+k+1$ of Lemma
\ref{oov2} and the assumption $t+k\ge 1113636=106\cdot 103\cdot 102$, which gives that $103(-1+\frac{t+k+1}{102} -202) \ge
t+k$.\end{proof}

\begin{lemma}\label{oov4}
Fix integers $\alpha \ge 0$, $u\ge 0$, and set $s:= t+k+1+2u$ and $\delta := a(s+2,t,k)_\alpha -a(s,t,k)_\alpha$. Assume $s+1
>\alpha$, $t\ge k \ge t/200$, $k\ge 4$  and $s \ge 1157520$. Let $e$ be the minimal integer $x$ such that $1\le x \le
\delta$ and
$b(s,t,k)_\alpha \le (x-1)(\delta -x-1)$. Then $e\le 104$.
\end{lemma}

\begin{proof}
The integer $e$ is defined by Lemma \ref{oov2}. Since $b(s,t,k)_\alpha \le s-2$, it is enough to check that $s-2 \le
103(\delta -202)$, which is true, because Lemma \ref{oov2} gives $\delta \ge -1 +\frac{s}{102}$
and we assumed $s \ge (2+103\cdot 106)\cdot 106 =1157520$ .
\end{proof}

\section{Proofs of $A(s,t,k)_\alpha$ and $A(s,t,k)$}\label{SprA}

 \begin{remark}\label{exp1}  Assume $b(s,t,k)_\alpha >0$. Assume that $A(s,t,k)_\alpha$ is true and take $(Y,Q,T_1)$ satisfying it. Set $\delta := a(s+2,t,k)_\alpha -a(s,t,k)_\alpha$. Let $e$ be the maximal positive integer such that $b(s-2,t,k) > (e-1)(\delta -e-1)$ and $e\le \delta/2$. Write $T_1 = R_1\cup \cdots \cup R_e\cup M_1\cup \cdots \cup M_{\delta -e}$ with $R_j\in |\Oo_Q(1,0)|$ and $M_h\in |\Oo _Q(0,1)|$. Now we modify $T_1$ to a new grid $T\in |\Oo_Q(e',\delta -e')|$, $e'\in \{e,e+1\}$, in which some of the lines of $T$ does not meet $Y\cap Q$. In each case we also  describe a very specific set $S\subseteq \mathrm{Sing}(T_1)$ with $\sharp (S) =b(s-2,t,k)_\alpha$ and show why we have $h^1(\Ii _{C_{t,k}\cup Y\cup S}(s)) =0$, $i=0,1$, although $S$ is not a subset of $\mathrm{Sing}(T_1)$. In each case we will say that $T$ is the grid adapted for $A(t,s,k)$. We have $\deg (Y\cup T)=a(s+2,t,k)_{\alpha}$. In all cases we will check that we have $\chi (\Oo _{Y\cup T\cup \chi}) = g(s+2,t,k)_\alpha -g_{t,k}$.

  \quad (a) Assume $(e-1)(\delta -e-1)+\alpha -e \le  b(s-2,t,k)_\alpha \le e(\delta -e-1)$. In this case we will have $e' = e+1$. Take distinct lines $L_i\in |\Oo _Q(0,1)|$, $1\le i\le \delta -e-1$, such that $L_i =M_i$ if  $\alpha-e\le  i \le b(s,t,k)_\alpha - (e-1)(\delta -e-1)$, while 
  $L_i\cap  Y =\emptyset$ in all other cases. Take $R_0\in |\Oo _Q(1,0)|$ containing a point of $Y\cap (Q\setminus T_1)$.
  We take $T:= R_0\cup \cdots \cup R_e\cup L_1\cup \cdots \cup L_{\delta -e-1}$. We take as $S$ the union of all points $R_j\cap L_i$ with
 either $j>1$ or $j=1$ and $1\le i \le b(s-2,t,k)_\alpha - (e-1)(\delta -e-1)$. Each $L_j$ moves in a family of lines of $Q$ with $M_j$ in its limit. In this degeneration of some of the lines of $T$ to some of the lines of $T_1$ the set $S$ degenerate to a subset
 $S_1\subseteq \mathrm{Sing}(T_1)$ (although $T$ and $T_1$ are grids with different bidegrees). By $A(s,t.k)_\alpha$ we have $h^i(\Ii _{C_{t,k}\cup Y\cup S_1}(s)) =0$, $i=0,1$. By the semicontinuity theorem we have $h^i(\Ii _{C_{t,k}\cup Y\cup S_1}(s)) =0$, $i=0,1$.
  We have $T\in |\Oo _Q(e+1,\delta -e-1)|$, $\sharp (\mbox{Sing}(T))
= (e+1)(\delta-e-1)$ and $\sharp (Y\cap T)) =1+\sharp (Y\cap (T\setminus R_0)) = 1+e+b(s,t,k) - (e-1)(\delta -e-1)-\alpha+1+e = b(s,t,k) -(e-1)\delta +e^2+2e -\alpha+1$.
Hence $1-\chi (\Oo _{Y'})= p_a(Y) -\sharp (S)+ \sharp (\mbox{Sing}(T\cup R_0)) +\sharp (Y\cap (T\cup R_0)) -\deg (R_0\cup T) = g(s,t,k) -b(s,t,k) +b(s,t,k)
- (e-1)\delta +e^2+2e-\alpha +1 + (e+1)(\delta-e-1)-\delta = g(s,t,k)_\alpha+\delta -\alpha = g(s+2,t,k)_\alpha$.

  \quad (b) Assume $(e-1)(\delta -e-1) +e-1 \le b(s-2,t,k)_\alpha \le (e-1)(\delta -e-1) +\alpha -1-e$. We have $(e-1)(\delta -e) \le b(s-2,t,k)_\alpha \le (e-1)(\delta -e) +\alpha -2$. We take
  $e'=e$ with $L_i=M_i$ if either $i> \alpha -e$ or  $1\le i \le b(s,t,k)_\alpha-(e-1)(\delta -e)$, while for the other indices $i$'s $L_i$ is a general deformation of $M$. We take as $S$ the union of all points $R_j\cap L_i$ with
 either $j>1$ or $j=1$ and $1\le i \le b(s-2,t,k)_\alpha - (e-1)(\delta -e)$. In this degeneration of some of the lines of $T$ to some of the lines of $T_1$ ($L_j$ degenerates to $M_j$) $S$ moves to a subset
 $S_1\subseteq \mathrm{Sing}(T_1)$. By $A(s,t,k)_\alpha$ we have $h^i(\Ii _{C_{t,k}\cup Y\cup S_1}(s)) =0$, $i=0,1$. By the semicontinuity theorem we have $h^i(\Ii _{C_{t,k}\cup Y\cup S_1}(s)) =0$, $i=0,1$. Set $Y':= Y\cup \chi \cup T$. We have $\deg (T) =\delta$, $\sharp (\mathrm{Sing}(T))
= e(\delta-e)$ and $\sharp (Y\cap T) =b(s,t,k)_\alpha-(e-1)(\delta -e) +\delta-\alpha+e =b(s,t,k)_\alpha -(e-2)\delta +e^2 -\alpha$. Thus $1-\chi (\Oo _{Y'}) = p_a(Y)-b(s,t,k)_\alpha +e(\delta-e)
+b(s,t,k)_\alpha-(e-2)\delta +e^2-\alpha-\delta = g(s,t,k)_\alpha + \delta -\alpha = g(s+2,t,k)_\alpha$.

 \quad ({c}) Assume $(e-1)(\delta -e-1) < b(s-2,t,k)_\alpha \le (e-1)(\delta -e-1) +e-2$ and hence $e\ge 3$. We have $b(s-2,t,k)_\alpha < (e-1)(\delta -e)$. 
Since $e\le \alpha -1$, Lemma \ref{oov2} gives $\delta \ge e$ and $\delta \ge 2\alpha -2 \ge 2e-4$ and hence $(e-1)(\delta -e-1) +1 \ge (e-2)(\delta -e) $. Thus $b(s-2,t,k)_\alpha \ge (e-2)(\delta -e)$. 
 Take $L_i=M_i$ if $1\le i \le \delta -e$ and $L_i$ a small deformation of $M_i$ not intersecting $Y\cap Q$. Let $S$ be the union of all points $R_j\cap L_i$ with either $j>1$ or $j=1$ and $1\le b(s,t,k)_\alpha - (e-2)(\delta -e)$. In this degeneration of some of the lines of $T$ to some of the lines of $T_1$ the set $S$ degenerates to a subset
 $S_1\subseteq \mathrm{Sing}(T_1)$. By $A(s-2,t.k)_\alpha$ we have $h^i(\Ii _{C_{t,k}\cup Y\cup S_1}(s-2)) =0$, $i=0,1$. By the semicontinuity theorem we have $h^i(\Ii _{C_{t,k}\cup Y\cup S_1}(s-2)) =0$, $i=0,1$. Set $Y':= Y\cup T\cup R_0\cup \chi$. We have $R_0\cup T\in |\Oo _Q(e,\delta-e)|$, $\deg (R_0\cup T) = \delta$, 
 $\sharp (\mathrm{Sing}(R_0\cup T))
= e(\delta -e)$ and $\sharp (Y\cap (R_0\cup T)) =1+\sharp (Y\cap T) =1+e-1 + b(s,t,k)_\alpha -(e-2)(\delta -e) -\alpha+e =b(s,t,k)_\alpha -(e-2)\delta +e^2 -\alpha$. Thus $1-\chi (\Oo _{Y'}) = p_a(Y)-b(s,t,k) +e(\delta-e)
+b(s,t,k)_\alpha-(e-2)\delta  -\alpha+e^2- \delta = g(s,t,k)_\alpha + \delta -\alpha = g(s+2,t,k)_\alpha$
\end{remark}

   From now on we take $\alpha = 202$ and write $a(s,t,k)$, $b(s,t,k)$ and $A(s,t,k)$ instead
 of $a(s,t,k)_{202}$, $b(s,t,k)_{202}$ and $A(s,t,k)_{202}$.

 \begin{lemma}\label{ov7}
 Assume $k \le t \le 200k$, $t+k > 102\cdot 229$ and $k^2\ge 207\cdot (k+2)$. 
 
 \quad (i) $A(t+k+1,t,k)$ is true.
 
 \quad (ii) Fix an integer $s\ge t+k+1$ such that $s\equiv t+k+1 \pmod{2}$ and
assume that $A(s,t,k)$ is true. Then $A(s+2,t,k)$ is true.
 \end{lemma}
 
 \begin{proof}
We first prove (ii). We will show in step (d) the small modification needed to get $A(t+k+1, t,k)$.
Set $\delta := a(s+2,t,k)-a(s,t,k)$. By assumption we have $t+k -1 \ge 102\cdot 229$. We take $Y$ of degree $d(s,t,k)$ and genus $g(t,s,k)$ such that for all $\tilde{S}$ in some grid and with cardinality $b(s,t,k)$
we have $h^i(\Ii _{C_{t,k}\cup Y\cup \tilde{S}}(s)) =0$, $i=0,1$. We will only use a very specific $\tilde{S}$ described separately in each case. In steps (a) and (b) we will construct
the curve $Y'$ appearing (as $Y$) in the statement of $A(s+2,t,k)$. In step ({c}) we will construct the grid $T_1$ such that for all $S'\subseteq \mathrm{Sing}(T_1)$ with $\sharp (S') =b(s+2,t,k)$  we have $h^i(\Ii _{C_{t,k}\cup Y'\cup S'}(s+2)) =0$, $i=0,1$.

 \quad (a) Assume $b(s,t,k) >0$.Take $e'\in \{e,e+1\}$ as in Remark \ref{exp1} so that $T$ is union of $e'$ lines $R_j\in |\Oo _Q(1,0)|$ and $\delta -e'$ lines $L_j\in |\Oo _Q(0,1)|$. Take as $\tilde{S}$ the set $S\subseteq \mathrm{Sing}(T)$, the one considered in Remark \ref{exp1}.
 Set $\chi := \cup _{o\in S} 2o$;  to make the construction of Remark \ref{exp1} we need that $2a(s,t,k) =\sharp (Y\cap Q)$ is at least the number of lines $R_j$ and $L_i$ containing a point of $Y\cap Q$; since
 the latter number is at most $\delta$, it is sufficient to quote Lemma \ref{oov5}.  
In all cases we have $Y':= Y\cup T\cup \chi$ with $T\in |\Oo _Q(e',\delta -e')|$ and $T\cap C_{t,k}=\emptyset$. Set $\Psi := Y\cap (Q\setminus T_1)$.

\quad {\emph {Claim 1:}} We have $2d_{t,k} + \sharp (\Psi )+ b(s+2,t,k) = (s+3-e')(s+3-\delta +e')$.

\quad {\emph {Proof of Claim 1:}} First assume that we are in case (a) of Remark \ref{exp1}. In this case we have $e'=e+1$ and $\sharp (\Psi ) = 2a(s,t,k) - b(s,t,k) +(e-1)\delta -e^2-2e +201$.
By (\ref{eqo4.1}) we have $(s+2-e)(s+4-\delta +e) =(s+3)^2 -1 -(s+2-e)\delta  -2e-e^2 = 2d_{t,k} +2a(s,t,k) + (s+1)\delta +202+ b(s+2,t,k)-b(s,t,k) -1
-(s+2-e)\delta -2e-e^2 =2d_{t,k} +\sharp (\Psi )+b(s+2,t,k)$.

Now assume that we are in case (b) of Remark \ref{exp1}. We have $e'=e$ and $\sharp (\Psi ) =2a(s,t,k) -b(s,t,k) +(e-2)\delta -e^2 +202$. By (\ref{eqo4.1}) we
have $(s+3-e)(s+3-\delta +e)
= (s+3)^2-(s+3-e)\delta -e^2 = 2d_{t,k} +2a(s,t,k) +(s+1)\delta +202+b(s+2,t,k)-b(s,t,k)-(s+3-e)\delta -e^2 =2d_{t,k} +b(s+2,t,k)+\sharp (\Psi )$.

Now assume that we are in case ({c}) of Remark \ref{exp1}. We have $a'=e$ and $\sharp (\Psi ) = 2a(s,t,k)-b(s,t,k) +(e-2)\delta -e^2 +202$. By (\ref{eqo4.1}) we have
$(s+3-e)(s+3-\delta +e) = (s+3)^2 -(s+3-e)\delta -e^2 = 2d_{t,k} +2a(s,t,k) +202 +(s+1)\delta +b(s+2,t,k) -b(s,t,k) -(s+3-e)\delta -e^2 = 2d_{t,k} +b(s+2,t,k)+\sharp (\Psi )$.\qed

By Claim 1 and the generality of $\Psi \cup (C_{t,k}\cap Q)$ we have $h^i(Q,\Ii _{C_{t,k}\cup Y'\cup S',Q}(s+2)) =0$. Since $\Res _Q(C_{t,k}\cup Y'\cup S') = C_{t,k}\cup Y\cup S$,
the residual sequence of $Q$ gives $h^i(\Ii _{C_{t,k}\cup Y'\cup S'}(s+2)) =0$, $i=0,1$.  
In step ({c}) we will prove that we may find $S'$ and a deformation of $Y'$ to get $A(s+2,t,k)$. 

\quad {\emph{Claim 2:}} $Y'$ is a flat limit of a family of connected smooth curves of degree $a(s+2,t,k)$ and genus $g(s+2,t,k)$.

\quad {\emph {Proof of Claim 2:}} Fix any two skew lines $D_1,D_2\subset \PP^3$
and any $p\in \PP^3\setminus (D_1\cup D_2)$. The linear projection from $p$ shows that there is a unique line $L$ with $p\in L$, $L\cap D_1\ne \emptyset$ and $L\cap D_2\ne \emptyset$. If $D_1$, $D_2$ and $p$ depend continuously from certain parameters, then the line $L$ depends continuously from the same parameters. 

We assume that $A(s,t,k)$ is in case (a) of Remark \ref{exp1} (the cases described in (b) and ({c}) of Remark \ref{exp1} are done in the same way). Take as a parameter space an integral affine curve $\Delta$ and fix $o\in \Delta$. Set $Y'(o):= Y'$,
$R_j(o):= R_j$ and $L_i(o):= L_i$. Take an algebraic family $\{R_j(z)\}_{z\in \Delta}$ of lines of $\PP^3$ with $R_j(z)$ transversal to $Q$ if $z\ne o$
and $R_j\cap Y\in R_j(z)$ for all $z$, and an algebraic family $\{L_i(z)\}_{z\in \Delta}$, $202-e\le i \le b(s,t,k) -(e-1)(\delta -e-1)$, of lines of $\PP^3$ with $L_i(z)$ transversal to $Q$ if $z\ne o$,
$L_i\cap Y\in L_i(z)$ for all $z$ and $L_i(z)\cap R_j(z)\ne \emptyset$ if and only if $j=0,1$. Changing if necessary $\Delta$ we may find an algebraic family $\{L_i(z)\}_{z\in \Delta}$, $1\le i \le 202 -1-e$,
of lines with $L_i(z)\cap R_j(z) \ne \emptyset$, $z\in \Delta \setminus \{o\}$, if and only $j=0$, and an algebraic family $\{L_i(z)\}_{z\in \Delta}$, $ b(s,t,k) -(e-1)(\delta -e-1) < i \le \delta -e-1$,
of lines with $L_i(z)\cap R_j(z)\ne \emptyset$ if any only if $j=0,1$. For any $z\in \Delta \setminus \{o\}$ set $Y'(z):= Y\cup \bigcup R_j(z)\cup \bigcup L_i(z)$.
The family  $\{Y'(z)\}_{z\in \Delta}$ is flat. Then we use Remark \ref{oo1} to smooth $Y\cup \bigcup R_j(z)\cup \bigcup L_i(z)$ for some $z\in \Delta \setminus \{o\}$\qed

\quad (b) Assume $b(s,t,k) =0$, i.e. $S=\emptyset$. Instead of lines $R_j$ and $L_j$ we take a line $R_0\in |\Oo _Q(1,0)|$ with $R_0\cap Y\ne \emptyset$
and $L_i\in |\Oo _Q(0,1)|$, $1\le i\le \delta-1$, such that $L_i\cap Y\ne \emptyset$ if and only if $i \ge 202$; we are using that $\delta \ge 202$ (Lemma \ref{oov5}).
We assume $L_i\cap C_{t,k}=\emptyset$ for all $i$. In this case we have $a_1=1$, $b_1= \delta -1$ and $\sharp (\Psi ) = 2a(s,t,k)
-\delta +201$. By (\ref{eqo4.1}) we have $(s+2)(s+4 -\delta ) = (s+3)^2 -1 -(s+2)\delta 
= 2d_{t,k} +2a(s,t,k) +(s+1)\delta+201+b(s+2,t,k) -(s+2)\delta = 2d_{t,k} +\sharp (\Psi ) +b(s+2,t,k)$. The union $Y'$ of $Y$
and all lines
$R_j$ and $L_i$ is smoothable by Remark \ref{oo1}.

\quad {\emph{Claim 3:}} Fix $S''\subset Q$ such that $\sharp (S'') = b(s+2,t,k)$ and $S''\cap Q\cap (T\cup Y\cup C_{t,k})=\emptyset$. If $h^1(Q,\Ii _{S'',Q}(s+2-e',s+2-\delta +e')) =0$,
then $h^i(\Ii _{C_{t,k}\cup Y'\cup S''}(s+2)) =0$, $i=0,1$.

\quad {\emph {Proof of Claim 3:}} We have $\mathrm{Res}_Q(C_{t,k}\cup Y'\cup S'') =C_{t,k}\cup Y' \cup S$ and $Q\cap (C_{t,k}\cup Y'\cup S'')
= (Q\setminus T)\cap (C_{t,k}\cup Y)\cup S''$. By $A(s,t,k)$ we have $h^i(\Ii _{C_{t,k}\cup Y'\cup S}(s)) =0$. Therefore the residual exact sequence of $Q$ shows that it is sufficient to prove
that $h^i(Q,\Ii _{(Q\setminus T)\cap (C_{t,k}\cup Y)\cup S''\cup T}(s+2,s+2)) =0$, $i=0,1$, i.e., to prove that we have $h^i(Q,\Ii _{(Q\setminus T)\cap (C_{t,k}\cup Y)\cup S''\cup T}(s+2-e',s+2-\delta+e')) =0$, $i=0,1$
for some $C_{t,k}$ and $Y$ for a fixed $S''$. We may deform $C_{t,k} \cup Y$ (keeping fixed $S''$) so that $(Q\setminus T)\cap (C_{t,k}\cup Y)$ are general. Thus it is sufficient
to observe that in parts (a) and (b) we proved that $\sharp ((Q\setminus T)\cap (C_{t,k}\cup Y))= (s+3-e')(s+3-\delta +e') -b(s+2,t,k)$.

\quad ({c})  If $b(s+2,t,k) =0$, then $S'=\emptyset$ and hence parts (a) and (b) prove $A(s+2,t,k)$, because we proved that $Y'$ is smoothable
(Claim 2 for the case $b(s,t,k)>0$) and Claim 3 with $S''=\emptyset$ gives $h^i(\Ii _{C_{t,k}\cup Y'}(s+2)) =0$, $i=0,1$. Now assume $b(s+2,t,k)>0$. We prove $A(s+2,t,k)$, but with exchanged the two rulings of $Q$.
Set $\tau := a(s+4,t,k)-a(s+2,t,k)$.  Let $f$ be the maximal positive integer
such that $(f-1)(\tau -f-1) < b(s+2,t,k)$. $A(s+2,t,k)$ is as in one of the cases (a), (b) or ({c}) of Remark \ref{exp1} with
$\tau$ instead of $\delta$ and $f$ instead of $e$. We call $T_1$ the grid of bidegree $(\tau -f',f')$, $f'\in \{f,f+1\}$ called
$T$ in in Remark \ref{exp1}, but with exchanged the $2$ rulings of $Q$, i.e. in all cases we take
$S'\subseteq
\mathrm{Sing}(T_1)$ with
$T_1$ unions of
$\tau -f'$ distinct lines of bidegree
$(1,0)$ and
$f'$ distinct lines of bidegree
$(0,1)$.  $h^1(Q,\Ii _{\mathrm{Sing}(T_1),Q}(s+2-e',s+2-\delta +e')) =0$. If $A(s+2,t,k)_\alpha$ is in case (a) (resp. (b) or ({c}) of Remark \ref{exp1} we take
$f'=f+1$ (resp.
$f' =f$).
We have $h^1(Q,\Ii _{\mathrm{Sing}(T_2),Q}(s+2-f',s+2-\delta +e')) =0$ (and hence $h^1(Q,\Ii
_{S',Q}(s+2-e',s+2-\delta +e')) =0$ for each $S' \subseteq \mathrm{Sing}(T_1)$), because $e'+\delta -e' \le s+2$ and $f'1+\delta \le s+2$, (we use that $e+1 +\tau \le
\tau +\alpha +1
\le s+2$ and $f+1+\delta -1\le s+2$ use Corollary \ref{C-boundDelta}). Claim 3 fives $A(s+2,t,k)$.

\quad (d) Now we prove $A(k+t+1,t,k)$. We have $h^i(\Ii _{C_{t,k}}(t+k-1))=0$. Since $g(t+k+1,t,k) =0$, we need to add a smooth rational curve of degree $a(t+k+1,t,k)$.
We start with a general $F\in |\Oo _Q(a(t+k+1,t,k)-1,1)|$. Thus $F\cap C_{t,k} =\emptyset$ and $F$ is a smooth rational curve. Since $C_{t,k}\cap Q$ is general in $Q$, for any set $S\subset Q\setminus (F\cup (C_{t,k}\cap Q))$ such that $\sharp (S) =b(k+t+1,t,k)$ and $h^1(Q,\Ii _S(t+k,t+k+2-a(t+k+1,t,k)+2)) =0$, the residual exact sequence of $Q$
gives $h^i(\Ii _{C_{t,k}\cup F\cup S}(t+k+1)) =0$, $i=0,1$. If $b(t+k+1,t,k)=0$, it is sufficient to deform $F$ to a general smooth rational curve $Y$ of degree $a(t+k+1,t,k)$
and use that $h^1(N_Y(-2)) =0$, because $a(t+k+1,t,k)\ge 3$. Now assume $b(t+k+1,t,k) >0$. We may take $S$ in a grid $T$ of bidegree $(e,\delta -e)$ as in the statement of $A(t+k+1,t,k)$, because we may deform $F$ to a curve transversal to $Q$ and fixing one point for each irreducible component of the grid $T$. We use that
$\delta \le a(t+k+1,t,k)$ by Lemma \ref{oov2.1}. \end{proof}

\begin{remark}\label{2no1}
The inductive proof of Lemma \ref{ov7} gives the following statement stronger that $A(s,t,k)$ but that we proved to be equivalent to it.
Set $\delta := a(s+2,t,k)_\alpha -a(s,t,k)_\alpha$. Let $e$ be the maximal positive integer such that $b(s,t,k) > (e-1)(\delta-e-1)$ and $e\le \delta /2$. Let $Q$ be a smooth quadric. Fix $C_{t,k}$ intersecting transversally $Q$ and such that $Q\cap C_{t,k}$ is formed by $2d_{t,k}$ general
 points of $Q$. There is a pair $(Y,T)$ with the following properties. $Y$ is a smooth and connected curve of degree $a(s,t,k)_\alpha$ and genus $g(s,t,k)_\alpha$ such
 that $Y\cap C_{t,k}=\emptyset$, $Y$ intersects transversally $Q$ and $(C_{t,k}\cup Y) \cap Q$ is general in $Q$ and  $T$ is a grid of $Q$ adapted to $(Y,C_{t,k}\cap Q)$
 such that for every $S\subseteq \mathrm{Sing}(T)$ we have $h^0(\Ii _{C_{t,k}\cup Y\cup S}(s)) =\max \{0,b(s,t,k)-\sharp (S)\}$ and $h^1(\Ii _{C_{t,k}\cup Y\cup S}(s)) =\max \{0,-b(s,t,k)+\sharp (S)\}$.
\end{remark}

 \section{The genus enters into the playground}\label{Sg}

 Now we fix $m, d$ and take $g:= 1+d(m-1)-\binom{m+2}{3}$. Recall that we take $\alpha =202$.
 
 \begin{remark}\label{v3}
Since $g =1+(m-1)d-\binom{m+2}{3}$ and $d< (m^2+4m+6)/4$, we have $g< (m^3+3m^2+2m-18)/12$. 
\end{remark}

From now on we assume $g \ge g_{1000,1000} = 1+ 1000\cdot 1001\cdot 1995/3$.
Let $t$ be the maximal integer such that $g_t\le \frac{999999}{1000000}g$. The maximality of $t$ gives
$g_{t+1} > \frac{999999}{1000000}g$ and so $\frac{999999}{1000000}g  -(t+1)(3t+1)/3  < g_t \le \frac{999999}{1000000}g$. Since $g\ge g_{1000,1000}$, we have $t\ge 1000$.
Let $k$ be the maximal positive integer such that $g_{t,k}\le g$ and $t+k \equiv m \pmod{2}$; $k$ exists, because $g_2 =2\le g/1000000 \le g-g_t$. Since $2g_t\ge 2\frac{999999}{1000000}g -2(t+1)(3t+1)/3 \ge g$, we have
$k\le t$. 
The minimality of $k$ gives $g_{t,k+2} > g \ge g_{t,k}$, i.e.
\begin{equation}\label{equ1}
g_{t,k} \le g \le g_{t,k} +2k^2+2k
\end{equation}
We use $A(s,t,k)$ for these integers $t,k$. A key tool that we need to cover all interval for the genus and the degrees in Theorems
\ref{i1}, \ref{i2} and \ref{i3} (and not just prove for many $(d,g)$ the existence
curves
$C$ with
$h^0(\Ii _C(m-1))=0$ and (degree, genus) = (d,g)), is that in (\ref{equ1}) the interval depends only on $k$ and it is quadratic
on $k$, while we take $t\gg k$, say $t\ge 30k$ (Lemma \ref{u3}).

Let
$y$ be the maximal integer
$s\ge t+k+1$, $s\equiv t+k+1\pmod{2}$ such that
$g(s,t,k) +g_{t,k}\le g$;
$y$ exists because $g(t+k+1,t,k) =0$, $g_{t,k}\le g$ and $g(s,t,k)>g(s-2,t,k)$ for all $s\ge t+k+3$ by Lemma \ref{oov2} and the
inequality $t+k+3\ge 202\times 102 +2$.  Note that $y\equiv m-1 \pmod{2}$. For all integers $x, y$ such that $x\ge y+2$ and $x\equiv y\pmod{2}$ we define the integers $u(x,t,k)$
and $v(x,t,k)$ by the relations

\begin{equation}\label{eqov3}
x(d_{t,k} +u(x,t,k))+3-g +v(x,t,k) =\binom{x+3}{3}, \ 0\le v(x,t,k) \le x-1
\end{equation}

The integers $y$, $u(x,t,k)$ and $v(x,t,k)$
 depend on $g$, but we do not put $g$ in their symbols.
 
 \begin{remark}\label{aaa+1+1}
 The main actor of this section is a a general smooth curve $Y$ of genus $g-g_{t,k}$ and of some degree $z\ge g-g_{t,k}+3$ with $z =u(x,t,k)$ for
 some $x\ge y+2$ with $x \equiv y \pmod{2}$. We need to check that $h^i(N_Y(-2)) =0$, $i=0,1$.
 If $g-g_{t,k} \ge 26$, then this is true by \cite[page 67, inequality $DP (g) \le g + 3$]{pe}. If $g-g_{t,k} = 0$, then we just quote \cite[Proposition 6]{ev}. Now assume
 $1\le g-g_{t,k} \le 25$. In our case we have $z=u(x,t,k) \ge u(y+2,t,k) \ge a(y,t,k) + 202\ge 202$ (Lemmas \ref{oov5} and \ref{oov6}). See \cite[page 106-107]{pe} for the best published results (we only need them for $z\ge 202$). Fix a general $A\subset Q$ such that $\sharp (A) =2a$. Since $h^1(N_C(-2)) =0$ for a general smooth curve
 of genus $g-g_{t,k}$ and degree $z$, we may find a smooth curve $Y$ of genus $g-g_{t,k}$ and degree $z$ intersecting transversally $Q$ and with $A =Y\cap Q$.
 \end{remark}
 
 For any integer $x\ge y+2$ with $x\equiv y \pmod{2}$ define the following Assertion $B(x,t,k)$.
 
 \quad {\bf {Assertion}} $B(x,t,k)$: Let $Q$ be a smooth quadric. Fix $C_{t,k}$ intersecting transversally $Q$ and such that $Q\cap C_{t,k}$ is formed by $2d_{t,k}$ general
 points of $Q$. Set $\delta:= u(x+2)-u(x,t,k)$. Let $e$ be the maximal positive integer such that $v(x,t,k) > (e-1)(\delta -e)$ and $e <\delta/2$.
We call $B(x,t,k)$ the existence of a pair $(Y,T)$ with the following properties:
 \begin{enumerate}
 \item $Y$ is a smooth and connected curve of degree $u(x,t,k)$ and genus $g-g_{t,k}$ such
 that $Y\cap C_{t,k}=\emptyset$, $Y$ intersects transversally $Q$ and $(C_{t,k}\cup Y) \cap Q$ is general in $Q$;
 \item  $T$ is a grid $T\in |\Oo _Q(e,\delta -e)|$;
 \item  for each $S\subseteq \mathrm{Sing}(T)$
 with $\sharp (S) =v(x,t,k)$ we have $h^i(\Ii _{C_{t,k}\cup Y\cup S}(x)) =0$, $i=0,1$.
 \end{enumerate}
 
 \begin{remark}
 We use Remark \ref{aaa+1+1} for the existence of $Y$.  Lemmas \ref{ov6+} shows that $e$ exists and that $e\le 201$.
 \end{remark}

\begin{remark} As in Remark \ref{2no1} proving inductively $B(x,t,k)$ we will also prove that for all $S\subseteq \mathrm{Sing}(T)$ we have $h^0(\Ii _{C_{t,k}\cup Y\cup S}(v)) =\max \{0,v(x,t,k)-\sharp (S)\}$ and $h^1(\Ii _{C_{t,k}\cup Y\cup S}(x)) =\max \{0,-v(x,t,k)+\sharp (S)\}$.
\end{remark}

 \section{Numerical lemmas, II}\label{SnII}
 In this section we collect the numerical lemmas related to section \ref{Sg} and used in the next sections. We take $m$, $d$
and $g$ as in section \ref{Sg} and in particular we are forced to assume $t\ge 1000$. 
 
 The next lemma only use that $g_t < g$.

\begin{lemma}\label{ov10}
If $t \ge 250$, then $m> 1.58t$.
\end{lemma}

\begin{proof}
We have $\sqrt[3]{4} \ge 1.587401$, $4 -1.58^3 = 0.055688$ and $3\cdot 1.58^2 =7.46892$. Remark \ref{v3} gives $g< 1 +(m^3+3m^2+2m-30)/12$. Assume $m \le 1.58t$. Since $g> g_t$, we get $1+(1.58^3t^3 +7.46892t^2 +3.16t-30)/12 > 1 + t(t+1)(2t-5)/6$, i.e. $13.46892t^2+13.16t  -30> 0.055688t^3$, which is false if $t \ge
250$.
\end{proof}

\begin{lemma}\label{u3}
If $t\ge 1000^5$, then $k\le t/30$.
\end{lemma}

\begin{proof}
Since $g_{t+1} > \frac{999999}{1000000}g$, we have $g_k \le g-g_t <g/1000000 + (t+1)(3t+1)/3 < g_{t+1}/999999 +(t+1)(3t+1)/3$. Assume
$k>t/30$. We have $30^3 =27000$. Even when $t/30\notin \NN$ we get $999999[6\cdot 27000 + t(t+30)(2t-150)] < 27000[6+(t+1)(t+2)(2t-3) + 2(t+1)(3t+1)]$,
which is false if $t\ge 100$.
\end{proof}

\begin{lemma}\label{u3.0}
If $t\ge 4000$, then $k> t/200$.
\end{lemma}

\begin{proof}
By (\ref{equ1}) we have $g_k \ge g/1000000 -2k^2-2k \ge g_t/999999 -2k^2-2k$. If $k \le t/200$ we get (even it $t/200 \notin \NN$, because $t/200 \ge 6$)
$1 + t(t+201)(2t-1000)/48000000 \ge 1/999999 + t(t+1)(2t-5)/5999994 -2t^2/40000 -2t/200$, a contradiction.
If $t \le 100k$ we get $999999(1 +k(k+1)(2k-5)/6) +1999998k^2+1999998k) \ge 1 + 100k(100k+1)(200k-5)/6$ which is false for
$k\ge 4000$. 
\end{proof}

From now on in this section we take $\alpha =202$.

\begin{lemma}\label{ov6+}
Assume $k\le t\le 200k$, $t+k\ge 102 \cdot 229$ and $y\ge 111\cdot 210$. Fix  an integer
$x\ge y+2$ such that $x\equiv y\pmod{2}$. Then $u(x+2,t,k)-u(x,t,k) \ge -2+\frac{x(x+1)}{102(x+2)} \ge 202$.
\end{lemma}

\begin{proof}
From (\ref{eqov3}) for $x+2$ and $x$ we get
\begin{align}\label{eqov4}
&2d_{t,k} +2u(x,t,k) +(x+2)(u(x+2,t,k) -u(x,t,k)) +\notag \\
&v(x+2,t,k)-v(x,t,k) =(x+3)^2
\end{align}
Since $u(x,t,k) \le a(x,t,k)$ and $v(x,t,k) \ge b(x,t,k)$ if $u(x,t,k) =v(x,t,k)$, (\ref{eqoov4}) and (\ref{eqov4}) give $u(x+2,t,k) -u(x,t,k) \ge (x+1)(a(x+2,t,k)-a(x,t,k))/(x+2) -1$. Use Lemma \ref{oov2}. The second inequality follows from the first one, because $x(x+1) \ge 102\cdot 204(x+2)$ if $x\ge 111\cdot 210$.
\end{proof}

\begin{lemma}\label{=o1}
Assume $k\le t\le 200k$, $t+k\ge 102 \cdot 229$. Then $u(y+2,t,k) - a(y,t,k)  \ge -2 + \frac{y(y+1)}{102(y+2)}\ge 202$.
\end{lemma}
\begin{proof}
Recall that $(y+2)a(y+2,t,k) +b(y+2,t,k) -g(y+2,t,k)-g_{t,k} = (y+2)u(y+2,t,k) +v(y+2,t,k) -g$ with $g(y+2,t,k) -g_{t,k} > g \ge g_{t,k}+g(y,t,k)$, $g(y+2,t,k) -g(y,t,k) =
a(y+2,t,k)-a(y,t,k) -202$ and hence $(y+1)(a(y+2,t,k) -a(y,t,k)) \le (y+2)(u(y+2,t,k) -u(y,t,k))$. Use Lemma \ref{oov2} to get the first inequality and the inequality $y\ge 111\cdot 210$
to get the second inequality.
\end{proof}

\begin{lemma}\label{ov6.1}
Assume $k\le t\le 200k$, $t+k\ge 102 \cdot 229$. Fix an integer $x\ge y+2$ such that $x\equiv
y\pmod{2}$ and $x\ge 210\cdot 111$. Set $\delta := u(x+2,t,k)-u(x,t,k)$. Let $e$ be the minimal integer $z$ with $(z-1)(\delta -z)\le
v(x+2,t,k)$. Then $e$ exists and $e\le 201$.
\end{lemma}

\begin{proof}Since $v(x+2,t,k) \le x+1$ and $\delta > -2 +\frac{x(x+1)}{102(x+2)} $ (Lemma \ref{ov6+}) it is
sufficient to check
that  $200 (-203 +\frac{x(x+1)}{102(x+2)}) \ge x+1$, i.e. $198x^2\ge (203\cdot 102\cdot 200 -200)x +400\cdot 203\cdot 102 +102$, which is true
if $x \ge 210\cdot 111$.
\end{proof}

\begin{lemma}\label{ov2+3}
Assume $x\ge y+2$, $k\le t\le 200k$, $t+k\ge 102 \cdot 229$ and $y \ge 111\cdot 210$. Fix an integer $x\ge y+2$ such that $x\equiv y \pmod{2}$. Then $2u(x,t,k) \ge
u(x+4,t,k)-u(x+2,t,k) +202$ and $2u(x,t,k)
\ge u(x+2,t,k)-u(x,t,k) +202$.
\end{lemma}

\begin{proof}
Assume $2u(x,t,k) \le u(x+2,t,k)-u(x,t,k)+201$. Since $v(x+2,t,k)\ge 0$ and $v(x,t,k) \le x-1$, (\ref{eqov4}) give 
\begin{equation}\label{eq=o1}2d_{t,k} +(2x+4)u(x,t,k) -201(2x+4)
-x+1\le (x+3)^2\end{equation} Since $(x+3) ^2 -(x+1)^2 =2x+8$, $201(2x+4)-x -201(2(x-2)+4) -x+2 = 802$ and $u(x,t,k) \ge u(x-2,t,k)+202$ by Lemma \ref{ov6+}, it is sufficient to disprove (\ref{eq=o1}) when $x=y+2$.
Since $u(y+2,t,k) \ge a(y,t,k) +202$ by Lemma \ref{=o1}, it is sufficient to prove that
\begin{equation}\label{eq=o2}
2d_{t,k} +(2y+8)a(y,t,k) > y^2+9y+22
\end{equation}
See the contradiction coming from the case $y=s$ of (\ref{eqoov4}).

Now assume $2u(x,t,k) \le u(x+4,t,k)-u(x+2,t,k) +201$. Since $v(x+4,t,k) \ge 0$, $v(x+2,t,k) \le x+1$ and $u(x+2,t,k) \ge
u(x,t,k)-2+\frac{x(x+1}{102(x+2)}$, the case $x':= x+2$ of (\ref{eqov4}) gives
\begin{equation}\label{eqov2+2}
2d_{t,k} + (2x+10)u(x,t,k) \le (x+5)^2+x+1+201(2x+4) +2(x+4) -\frac{x(x+1)(x+4)}{102(x+2)}
\end{equation}
As in the first part of the proof we reduce to prove an inequality weaker than (\ref{eq=o2}).
\end{proof}

\begin{lemma}\label{=o2}
Assume $x\ge y+2$, $k\le t\le 200k$, $k\le t \le 200k$, $t+k\ge 102 \cdot 229$ and $y \ge 111\cdot 210$. Then $2a(y,t,k) \ge u(y+2,t,k)-a(y,t,k)+202$ and
$2a(y,t,k) \ge u(y+4,t,k) -u(y+2)+202$.
\end{lemma}

\begin{proof}
The first inequality is true by Lemma \ref{oov5}, because $u(y+2,t,k) \le a(y+2,t,k)$. From ({eqov4}) for $x':= x+2$ and (\ref{eqo4.1}) for $s=y+2$ we get
$(y+4)(u(y+4,t,k)-u(y+2,t,k)) + v(u+4,t,k)-v(u+2,t,k) = (y+4)(a(u+4,t,k)-a(u+2,t,k)+b(u+4,t,k)-b(u+2,t,k) +g(y+4,t,k)-g(y+2,t,k)$.
Use Lemma \ref{oov5} and that $g(y+4,t,k)-g(y+2,t,k) \ge -106 +(y+2)/102$.
\end{proof}

\begin{lemma}\label{=o3}
Assume $k\le t\le 200k$, $t+k\ge 102 \cdot 229$ and $y \ge 211\cdot 210$. Set $\delta:= u(y+2,t,k)-a(y,t,k)$ and $\tau := u(y+4,t,k)-u(y+2,t,k)$. Let $e$ (resp $f$)
be the minimal positive integer such that $(e-1)(\delta -e)\le v(y+2,t,k)$ (resp. $(f-1)(\tau -f) \le v(y+4,t,k)$). Then $e\le 201$ and $f\le 201$.
\end{lemma}

\begin{proof}
The assertion on $f$ is true by the case $x=y+2$ of Lemma \ref{ov6.1}. Since $v(y+2,t,k) \le y+1$, to prove the assertion on $e$ it is sufficient to prove
that $200\delta \ge 200\cdot 201 + y+1$. By Lemma \ref{=o1} it is sufficient to prove that $ -400 + 200\frac{y(y+1)}{102(y+2)} \ge 200\cdot 201 +y+1$, i.e.
$200y(y+1) \ge 200\cdot 203\cdot 102(y+2) + 102y(y+1)$, contradicting our assumption on $y$.
\end{proof}

\begin{lemma}\label{ov9}
If $k\le t\le 200k$ and $t\ge 3\cdot 10^4$, then $y\le m-7$.
\end{lemma}

\begin{proof}
Since $y\equiv m-1 \pmod{2}$, it is sufficient to prove that $y\le m-6$. Lemma \ref{ov10} gives $m > 1.58t$. Thus it is sufficient to prove that $y-t-k -1 \le 0.58t -k-7$. Assume $y-t-k-1 >0.58t-k-7$. Recall that for $s=t+k+3,\dots ,y$ we have $g(s,t,k) -g(s-2,t,k) \ge
-1 +s/102 $ (Lemma \ref{oov2}). Hence $g(y,t,k) \ge -202(y-t-k-1)/2 +(y-t-k-1)(y+t+k+3)/204 =(y-t-k-1)(y+t+k-10707)/204>
(0.58t-k-8)(1.58t-10714)/204$.   Since $g(y,t,k)  \le 2k^2+2k$ by (\ref{equ1}), we get
$(0.58t-k-8)(1.58t-10714) < 408k(k+1)$. Lemma \ref{u3} gives $t\ge 30k$. Hence $k\le t/30$. We get
$30(17.4t-8)(1.58t -10714) < 408t(t+30)$, which is false if $t \ge 3\cdot 10^4$.
\end{proof}

 \section{Proof $B(x,t,k)$}\label{SprB}

\begin{lemma}\label{ov12}
Fix integers $t, k$ such that $k\le t\le 200k$, $t+k\ge 102 \cdot 229$ and $k^2\ge 207(k+2)$. Fix an integer $x\ge
y+2$ such that $x\equiv y \pmod{2}$. If
$B(x,t,k)$ is true, then $B(x+2,t,k)$ is true.
\end{lemma}

\begin{proof}
Fix $Q$, $C_{t,k}$, and $(Y,T_1)$  satisfying $B(x,t,k)$ with (if $v(x,t,k)>0$) a grid $T_1$. Set $\delta := u(x+2,t,k)-u(x,t,k)$.

\quad (a) Assume $v(x,t,k)=0$. Fix $p\in Y\cap Q$ and take a general $E\in |\Ii _{p,Q}(1,\delta-1)|$. We have $E\cap C_{t,k} =\emptyset$ and $E\cap Y =\{p\}$.
Since $E$ intersects almost transversally $Y$ and at a unique point, the nodal curve $Y\cup E$ is smoothable (\cite{hh}, \cite{s}) and $p_a(Y\cup E) =g-g_{t,k}$. Fix any $S'\subset Q\setminus E$ such that $\sharp (S')=v(x+2,t,k)$,
$S'$ contains no point of $(C_{t,k}\cup Y)\cap Q$. We assume that $h^1(Q,\Ii _{S',Q}(x+1,x+3-\delta))=0$. Set $\Sigma := ((C_{t,k}\cup Y)\cap Q)\setminus \{p\}$.
Since $\Sigma$ is general in $Q$ and $\sharp (\Sigma \cup S') = (x+2)(x+4-\delta)$ by (\ref{eqov4}) and the equality $(x+3)^2 -(x+2)(x+4) = 1$, we have
$h^i(Q,\Ii _{E\cup \Sigma \cup S',Q}(x+2,x+2)) =0$, $i=0,1$. Since $h^i(\Ii _{C_{t,k}\cup Y}(x)) =0$, $i=0,1$,
the residual exact sequence of $Q$ gives $h^i(\Ii _{C_{t,k}\cup Y\cup E\cup S'}(x+2)) =0$, $i=0,1$.

\quad (b) Assume $v(x,t,k) >0$. Let $e$ be the maximal positive integer such that $v(x,t,k) > (e-1)(\delta -e)$ and
$e< \delta /2$. Lemma \ref{ov6.1} gives that $e$ exists and that $e\le 201$. The definition of $e$ gives $v(x,t,k) \le e(\delta
-e-1) $. The grid $T_1$ has $e$ lines $R_j\in |\Oo _Q(1,0)|$ and $\delta -e$ lines $M_i\in |\Oo _Q(0,1)|$. Each irreducible component of $T_1$ contains a point of $Y\cap Q$.
We deform $T_1$ to a grid $T\in |\Oo _Q(e,\delta -e)|$ with the same irreducible component $R_j$ of bidegree $(1,0)$ and with $\delta -e$ irreducible component $L_i$ of bidegree $(0,1)$ with $L_i = M_i$ for $1\le i\le v(x,t,k)-(e-1)(\delta -e)$, while the other lines $L_i$ are small deformations of $M_i$ not intersecting $Y\cap Q$. By $B(x,t,k)$ the set $S$ is the union of all points $R_j\cap L_i$
 with either $j\ge 2$ or $j=1$ and $1\le i \le v(x,t,k) - (e-1)(\delta -e)$; we have enough points of $Y\cap Q$ to link all lines $R_j$ and
 the prescribed lines $L_i$, because we only need $2u(x,t,k) \ge \delta$ and we use Lemma \ref{ov2+3}. Let $T$ be the union of all lines $R_j$ and $L_i$. We have $\sharp (\mathrm{Sing}(T))=e(\delta -e)$. Set $\chi := \cup _{o\in S} 2o$
 and $Y':= Y\cup T\cup \chi$. We have $\deg (Y'_{\red}) = u(x+2,t,k)$ and $\chi (\Oo _{Y'}) = \chi (\Oo _Y) +\deg (T) +\sharp (S) -\sharp (\mathrm{Sing}(T))+\sharp (T\cap Y) =
 1-p_a(Y) = 1-g+g_{t,k}$. We have $\Res _Q(C_{t,k} \cup Y') = C_{t,k}\cup Y\cup S$. Fix any $S'\subset Q\setminus T$ containing no point of $Q\cap (C_{t,k}\cup Y)$, with $\sharp (S') =v(x+2,t,k)$ and with
 $h^1(Q,\Ii _{S',Q}(x+2-e,x+2-\delta+e)) =0$. Since $(x+3)^2 -(x+3-e)(x+3-\delta+e) = v(x+2,t,k) +2d_{t,k} +2u(x,t,k) -\sharp (T\cap Y)$
 by (\ref{eqov4}) and $(Y\cap C_{t,k})\cap Q$ are general in $Q$,
 we have $h^i(Q,\Ii _{S'\cup (Y'\cup C_{t,k})\cap Q,Q}(x+2,x+2)) =0$, $i=0,1$. The residual sequence of $Q$ gives $h^i(\Ii _{C_{t,k}\cup Y'\cup S'}(x+2,x+2)) =0$, $i=0,1$. 
We claim that $Y'$ is a flat limit of a family of smooth and connected curves of
  genus $g-g_{t,k}$ and degree $u(x+2,t,k)$. 
  Take a smooth and connected affine curve $\Delta$ and $o\in \Delta$. Set
  $\{o_j\}:= R_j\cap Y$, $1\le j\le e$, $\{q_i\}:= L_i\cap Y$, $1\le i \le v(x,t,k) - (e-1)(\delta -e)$, and $\{e_h\}:= R_1\cap L_h$, $v(x,t,k) - (e-1)(\delta -e) < h\le \delta -e$. 
  We may find flat families $\{R_j(z)\}_{z\in \Delta}$, $1\le j\le e$, and $\{L_i(z)\}_{z\in \Delta }$, $1\le j \le \delta-e$, of lines with the following properties. For all $i, j$ and all $z\in \Delta \setminus \{o\}$
 we have $R_j(o)=R_j$, $R_j(z)\nsubseteq Q$, $o_j\in R_j(z)$, $L_i(o)=L_i$, $L_i(z)\nsubseteq Q$, $q_i\in L_i(z)$ if $1\le i \le v(x,t,k) - (e-1)(\delta -e)$, $L_i(z)\cap R_j(z)=\emptyset$ for all $j$ if  $1\le i \le v(x,t,k) - (e-1)(\delta -e)$ and $L_i(z)\cap R_j(z)\ne \emptyset$, $v(x,t,k) - (e-1)(\delta -e)<i\le \delta -e$, if and only if $j=1$.
  For all $z\in \Delta \setminus \{o\}$ set $Y'_z:= Y\cup \bigcup R_j\cup \bigcup L_i$. The family $\{Y'_z\}_{z\in \Delta}$ is flat. Each $Y'_z$, $z\in \Delta \setminus \{o\}$
  is smoothable (Remark \ref{oo1}).
  Hence $Y'_z$, $z\ne o$, is a flat limit of a family of smooth curves of genus $g-g_{t,k}$ and degree $u(x+2,t,k)$ (Remark \ref{oo1}). Hence there is a flat family
  $\{Y''_z\}_{z\in \Gamma}$, with $\Gamma$ an integral affine curve, $o\in \Gamma$, $Y''_o =Y'$, and $Y''_z$ smooth and of genus $g-g_{t,k}$ for all $z\ne o$. 
 
 \quad ({c}) Now we check that we may take $S'$ in steps (a) and (b) to get a solution of $B(x+2,t,k)$, except that if $v(x+2,t,k) >0$ we shift the two rulings of $Q$. 
 Assume $v(x+2,t,k) >0$ and take $Y$, $S$, $e$, $R_j$, $L_i$ as in step (b). Set $\tau:= u(x+4,t,k) -u(x+2,t,k)$. Let $f$ be the maximal positive integer such that $v(x+2,t,k) > (f-1)(\tau -f)$ and $f<\tau /2$.
Lemma \ref{ov6+} gives that $f$ exists and that $f\le 201$.  Note that $v(x+2,t,k) \le f(\tau -f-1) $. Fix distinct lines
$R'_j\in |\Oo _Q(0,1)|$, $1\le j \le f$, each of them containing a point
 of $Y\cap (Q\setminus T)$. Fix distinct lines $L'_i\in |\Oo _Q(1,0)|$, $1\le i \le \tau-f$, such that $L'_j$ contains a point of $Y\cap Q$ if and only if $1\le i \le v(x+4,t,k) - (f-1)(\tau -f)$. We impose that no line $L'_i$ contain a point of $C_{t,k}\cap Q$. We impose $R'_j\ne L_h$ and $L'_i \ne R_z$ for all $i,j,h,z$. So if $p\in Y\cap Q$ we allow
 that it is contained in a line of $T$ and in one of the lines $R'_j$, $L'_i$, but in this case we assume that they are in different rulings of $Q$. So we may assume that $R_j\ne L'_i$
 and $R'_j\ne L_i$ for all $i,j$.
 Since $\tau \ge \delta$ (use (\ref{eqov4} for $x$ and the integer $x':= x+2$) and $e\le 201$, it is sufficient to use that
 $\sharp (Y\cap Q) =2u(x,t,k) \ge \tau +201$ and that $2u(x,t,k) \ge \delta +201$ (Lemma \ref{ov2+3}). We take as $S'$ the
union of all points $R'_j\cap L'_i$ with either $j>1$ or $j=1$ and $1\le i \le v(x+2t,k) - (f-1)(\tau -f)$.
 We have $h^1(Q,\Ii _{S',Q}(x+2-e,x+2-\delta+e)) =0$, because $\tau-1 \le x+3-e$ (since $\tau \le x-197$ by Lemma \ref{=o3}).
and $f\le 201
 \le x+2-\delta+e$ (Lemma \ref{ov6+}). Let $N'$ (resp. $N''$) be the set of all points of $Y\cap Q$ contained in some line
$R'_j$ (resp. $L'_i$). Set $N:= N'\cup N''$. Note that $N'\cap N''=\emptyset$.
 Now take the last deformation made in step (b) with $\Gamma$ as its parameter space. Since $Y$ is transversal to $Q$, restricting $\Gamma$ to a neighborhood
 of $o$ and then taking a finite covering we may assume that $\{Y''_z\}_{z\in \Gamma }$ has $\sharp (N)$ sections $m_p$, $p\in N$, with $m_p(o) =p$ and $m_p(z)\in Q\cap Y''_z$
 for all $z$. For each $z\in \Gamma$ and any $p\in N'$, say $p\in Y\cap R'_j$ (resp $p\in N''$, say $p\in L'_i\cap Y$), let $R'_j(z))$ (resp. $L'_i(z)$)
 be the line of bidegree $(0,1)$ (resp. $(1,0)$) containing the point $m_p(z)$. If $L'_i\cap Y=\emptyset$, then set $L'_i(z):= L_i$. Taking the union of all these lines
 we get a  grid $T_1(z)\in |\Oo _Q(\tau-\sharp (N'),\sharp (N')|$  union
 of $\sharp (N')$ lines of bidegree $(0,1)$ containing a point of $Y\cap Q$ and $\tau -\sharp (N')$ lines of bidegree $(1,0)$, $\sharp (N'')$ of them containing
 a point of $Y\cap Q$. Set $S_o:= S'$. For each $z\in \Gamma \setminus \{o\}$ we get a set ${S'}_z\subset Q$ taking the union
 of all $R'_j(z)\cap L'_i(z))$ according to the rules of the cases $(f,1)$, $(f,0,+)$ or $(f,0,-)$. Note that $\sharp ({S'}_z)=v(x+2,t,k)$. The family $\{{S'}_z\}_{z\in \Delta}$ is flat.
The semicontinuity theorem for cohomology gives $h^i(\Ii _{C_{t,k}\cup Y''_z\cup {S'_z}}(x+2))=0$, $i=0,1$, for a general $z\in \Gamma$.

 If $v(x,t,k)=0$, then we take $S=\emptyset$. In this case the same construction work, because $\tau\le x+2$ and $f \le x+3-u(x+2,t,k)+u(x,t,k)$ and $2u(x-2,t,k) \ge \delta$ (Lemma \ref{ov2+3}).
\end{proof}

\begin{lemma}\label{ov11}
Assume $k\le t\le 200k$, $t +k \ge 42040$, that $A(y,t,k)$ is true and that either $y=t+k+1$ or $A(y-2,t,k)$
is true. Then
$B(y+2,t,k)$ is true.
\end{lemma}

\begin{proof}
Remember that $y\ge t+k+1$ and that $y\equiv t+k+1\pmod{2}$. 

 First assume $y\ge t+k+3$. Hence $A(y-2,t,k)$ is true. Taking the difference between (\ref{eqov3}) for $x=y+2$ and (\ref{eqo5}) for $s=y$
we get
\begin{align}\label{eqov11.1}
& 2(d_{t,k} +a(y,t,k)) + (y+2)(u(y+2,t,k)-a(y,t,k)) \notag \\
& +v(y+2,t,k)-b(y,y,k) -g+g_{t,k} +g(y,t,k) = (y+3)^2
\end{align}
Set $\delta := u(y+2,t,k)-a(y,t,k)$, $\gamma := g-g_{t,k}-g(y,t,k)$ and $\mu:= \delta -\gamma $. Since $g\le
g_{t,k}+g(y+2,t,k)$ we have $a(y+2,t,k)\ge u(y+2,t,k)$. By the definition of $u(y+2,t,k)$ we get
\begin{align}\label{eqov11.1=}
&(y+2)(a(y+2,t,k)-u(y+2,t,k)) = \\
&v(y+2,t,k)-b(y+2,t,k) + g_{t,k}+g(y+2,t,k) -g \notag
\end{align}

\quad \emph{Claim:} We have $\delta -\gamma \ge 201$.

\quad \emph{Proof of the Claim:} The integers $u(y+2,t,k)$,$v(y+2,t,k)$, $\delta$, $\gamma$ and $\mu$ depend on $g$ and we
write
$u(y+2,t,k)(g)$, $v(y+2,t,k)(g)$, $\delta (g)$ and $\mu (g)$ to stress their dependence on $g$. They are defined for all $g$
with
$g_{t,k}+g(y,t,k)\le y< g(y+2,t,k)$ by the definition of $y$, but we may also define them for $g=g_{t,k}+g(y+2,t,k)$,
writing $u(y+2,g,t,k)(g_{t,k}+g(y+2,t,k)): = a(y+2,t,k)$ and $v(y+2,g,t,k)(g_{t,k}+g(y+2,t,k)):= b(y+2,t,k)$ and still have
(\ref{eqov11.1}) and (\ref{eqov11.1=}). The Claim is true with strict inequality for $g= g_{t,k}+g(y+2,t,k)$ by Lemma
\ref{=o1}. Use (\ref{eqov11.1=}) both for $g$ and $g-1$. When we decrease by one the genus in these two equations we decrease
by at most one the integer $u(y+2,t,k)(g)$. Hence we get the Claim.

We introduce the following assertion.

\quad {\bf {Assertion}} $A'(y,t,k)$: There is a pair $(Y,T_1)$ with the following properties. $Y$ is a smooth and connected curve of degree $a(y,t,k)$ and genus $g(y,t,k)$ such
 that $Y\cap C_{t,k}=\emptyset$, $Y$ intersects transversally $Q$ and $(C_{t,k}\cup Y) \cap Q$ is general in $Q$. In particular no line of $Q$ contains two or more points
 of $Y\cup C_{t,k}$. If $b(y,t,k) =0$ we take as $T$ any grid. 
 Now assume $b(y,t,k) > 0$. Let $e$ be the maximal positive integer such that $b(y,t,k) > (e-1)(\delta-e-1)$ and $e<\delta /2$. Lemma \ref{=o3} gives that $e$ exists and that $e\le 201$. We take a grid $T\subset Q$ of bidegree $(e,\delta -e)$ adapted to $(Y,C_{t,k}\cap Q)$ such that $h^i(\Ii _{C_{t,k}\cup Y\cup S}(y)) =0$, $i=0,1$, for each $S\subseteq \mathrm{Sing}(T_1)$ with $\sharp (S) =b(y,t,k)$.
 
 Assertion $A'(y,t,k)$ is proved as in Lemma \ref{ov7}, with the distinction of the 3 cases introduced in Remark \ref{exp1}, except for the following modification
 in one of them, which we now discuss as step (a).
 
 \quad (a)  Assume $(e-1)(\delta -e-1) < b(y,t,k) \le (e-1)(\delta -e-1) +e-2$ and $\delta <\mu +e$. Since $e\le 202$, we get $\gamma \le 201$. Let $f$ be the maximal integer
 such that $f(\delta -f) \le \gamma + b(y,t,k)$ and $f\le \delta /2$. We say that $A'(y,t,k)$ is in case $(f,2)$. We assume the existence of distinct lines $R_j\in |\Oo _Q(1,0)|$, $1\le j\le f$, with $R_j\cap (Y\cap Q)\ne \emptyset$ for all $j$ and distinct lines $L_i\in |\Oo_Q(0,1)|$,
 $1\le i \le \delta -f$,
 such that $L_i\cap (Y\cap Q) \ne \emptyset$ if and only if $1\le i \le b(y,t,k)+\gamma -f(\delta -f)$,
$L_i\cap (C_{t,k}\cap Q) =R_j\cap (C_{t,k}\cap Q)=\emptyset$ for all $i, j$, and let $S$ be the union of all points $R_j\cap L_i$
 with either $j\ge 2$ or $j=1$ and $1\le i \le b(y,t,k) - f(\delta -f)$.
 
By Lemma \ref{=o2} we have enough points of $\sharp (Y\cap Q)\ge \delta$  to find the lines $R_j$ and $L_i$
 as described in Remark \ref{exp1}; also we use this observation and the same in the analogous of step ({c}) of the proof of Lemma \ref{ov11}.

\quad (b) Note that in the case  we have $\gamma = g-g(y,t,k)-g_{t,k}\le 201$. In the case $\gamma =0$, the proof of Lemma \ref{ov11} (case ({c}) of Remark \ref{exp1})
would work verbatim, while in the case $0< \gamma \le 201$ it only requires the modifications outlined in (d), which explains
exactly which lines $L_i$ must intersect $Y\cap Q$. Part ({c}) of the proof of Lemma \ref{ov12} with $x=y-2$ shows how to
prove that $A(y-2,t,k)$ implies $A'(y,t,k)$. As in Remark \ref{exp1} and Lemmas \ref{ov12} and \ref{ov11} $A'(y,t,k)$
implies $B(y+2,t,k)$. 
\end{proof}

\section{Proofs of Theorems \ref{i1}, \ref{i2}, and \ref{i3}}\label{Se}
\begin{proof}[End of the proof of Theorem \ref{i2}:] Lemma \ref{ov9} gives $y\le m-7$. Hence $B(m-5,t,k)$ and $B(m-3,t,k)$ are true.  Since $1+(m-1)d +1-g =\binom{m+1}{3}$, we have $u(m-1,t,k) = d-1$ and $v(m-1,t,k) = m-3$. Take a solution of $B(m-5,t,k)$ with respect to $Q$ and use the proof of Lemma \ref{ov12} to prove the existence of a solution of the following modification $B'(m-3,t,k)$
of $B(m-3,t,k)$:

\quad {\bf {Assertion}} $B'(m-3,t,k)$: Let $Q$ be a smooth quadric. Fix $C_{t,k}$ intersecting transversally $Q$ and such that $Q\cap C_{t,k}$ is formed by $2d_{t,k}$ general
 points of $Q$. We call $B'(m-3,t,k)$ the existence of a pair $(Y,T_1)$ with the following properties. $Y$ is a smooth and connected curve of degree $u(m-3,t,k)$ and genus $g-g_{t,k}$ such
 that $Y\cap C_{t,k}=\emptyset$, $Y$ intersects transversally $Q$ and $(Y\cup C_{t,k}) \cap Q$ is general in $Q$. In particular no line of $Q$ contains two or more points
 of $Y\cup C_{t,k}$. $T_1$ is a grid adapted to $(Y,C_{t,k}\cap Q)$, which we now describe. In all cases we assume that for all $S\subseteq \mathrm{Sing}(T_1)$
 with $\sharp (S) =v(m-3,t,k)$ we have $h^i(\Ii _{C_{t,k}\cup Y\cup S}(x)) =0$, $i=0,1$. If $v(m-3,t,k)=0$, then
take as $T_1$ any adapted grid. Now assume $v(m-3,x,t) >0$. Set $\delta := d-u(m-3,t,k)$. Let $e$ be the maximal positive integer such that $v(x,t,k) > (e-1)(\delta -e)$ and $e\le \delta/2$. Since $d= u(m-1,t,k)+1$, Lemma \ref{=o3} gives that $e$ exists and that $e\le 201$. We assume that the grid $T_1$ has bidegree $(e,\delta -e)$.

If $v(m-3,t,k)>0$ we write $R_j$, $1\le j \le e$, for the lines of bidegree $(1,0)$ of $T_1$ and $M_i$, $1\le i\le \delta -e$, for the one of bidegree $(0,1)$. As in the proof of Lemma \ref{ov7} we deform $T_1$ to another grid $T$ with the same lines of bidegree $(1,0)$ and with lines $L_i$ of bidegree $(0,1)$ only some of them containing a point of $Y\cap Q$.

The definition of $e$ gives $v(m-3,t,k) \le e(d -u(m-3,t,k) -e-1)$. Fix distinct lines $R_j\in |\Oo _Q(1,0)|$, $1\le j\le e$, with $R_j\cap (Y\cap Q)\ne \emptyset$ for all $j$ and distinct lines $L_i\in |\Oo_Q(0,1)|$,
 $1\le i \le \delta -e$,
 such that $L_i\cap (Q\cap (Y\cup C_{t,k})) \ne \emptyset$ if and only if $1\le i \le v(x,t,k) -(e-1)(\delta -e)$. We assume $Y\cap R_j\cap L_i=\emptyset$ for all $i, j$.  We assume $R_j\cap Y=\emptyset$ for all $j$ , $L_i\cap C_{t,k} =\emptyset$ if $i\le \delta -e-2$, $L_i\cap C_t\ne \emptyset$ if and only if $i=\delta -e-1$
 and $L_i\cap C_k\ne \emptyset$ if and only if $i=\delta -e$. Let $S$ be the union of the points $R_j\cap L_i$
 with either $j\ge 2$ or $j=1$ and $1\le i \le v(x,t,k) - (e-1)(\delta -e)$. About $B'(m-3,t,k)$ we only use that $h^i(\Ii _{C_{t,k}\cup Y\cup S}(x)) =0$, $i=0,1$, for this specific set $S\subset \mathrm{Sing}(T)$. Set $\chi := \cup _{o\in S} 2o$. Let $Y'$ be the union of $Y$, $\chi$ and all lines $R_j$ and $L_i$.
 $Y'$ is smoothable to a smooth and connected curve $Y''$ of genus $g-g_{t,k}$ and we may find a smoothing family fixing the two points of $Y'\cap C_{t,k}$. For this choice of $Y''$
 the curve $M:= Y''\cup C_{t,k}$ is a nodal and connected curve with arithmetic genus $g$ and with exactly $2$ nodes. Since $t\ge k$, we have $h^1(\Oo _M(t)) =0$ and hence $h^1(\Oo _M(m-2)) =0$. The vector bundle $N_M|C_t$ (resp. $N_M|C_k$, resp.
 $N_M|Y''$) is obtained from $N_{C_t}$ (resp. $N_{C_k}$, resp. $N_{Y''}$) making a positive elementary transformation at the point $Y''\cap C_t$ (resp. the point $Y''\cap C_k$, resp. each of the two points $Y''\cap C_{t,k}$) in the  direction corresponding to the tangent line of $Y''$ (resp. $Y''$, resp. $C_{t,k}$)
 at the point $Y''\cap C_t$ (resp. the point $Y''\cap C_k$, resp. each of the two points of $Y''\cap C_{t,k}$). Since
 $h^1(N_{Y''}(-2)) =0$, $h^1(N_{C_{t,k}}(-2))=0$ and $\sharp (\mathrm{Sing}(M)) =2$, the Mayer-Vietoris exact sequence of $Y''$ and $C_{t,k}$ gives
 $h^1(N_M(-1)) =0$. Hence $M$ is smoothable (\cite[Corollary 1.2]{fl1}). By the semicontinuity theorem for cohomology a smoothing of $M$ proves Theorem \ref{i2} for the pair $(d,m)$, except that we must discuss the bounds on $m$ and $g$ assumed in Theorem \ref{i2}. We need $g, m$ for which we may take $t\ge 10^5$ and some $k$
 with $t/30 \ge k\ge t/200$ (these bounds are sufficient to use Remarks \ref{ov8}  and \ref{aaa+1+1}). For $m$ we need $m\ge k+t+7$ and we do not need other assumptions on $m$ if the pair $(g,m)$ allows us to do the construction, i.e. the burden is shifted to $g$.
It is sufficient to have $g\ge g_{t,k}+g(t+k+1,t,k)$. We have $g_x = 1+x(x+1)(2x-5)/6 \le x^3/3$ for all $x\ge 10$. Since $k\le
t/30$ for very large $t$ (Lemma \ref{u3}), we assume $g\ge g_{10^5}+ g_{\lfloor 10^4/3\rfloor}$. Hence it is
sufficient to assume $g\ge 10^{15}/3 +10^{12}/27$. Hence it is sufficient to assume $g \ge 0.34\cdot 10^{15}$.
 \end{proof}

\begin{proof}[Proof of Theorem \ref{i1}:] Take $(d,m)$ in the range A and set $g:= 1+m(d-1)-\binom{m+2}{3}$. By
\cite[Corollary 2.4]{bbem} we may
 assume $d<\frac{m^2+4m+6}{4}$. Since $m\ge  13.8\cdot 10^5$, and $d> \frac{m^2+4m+6}{6}$, we have $d \ge 31.3\cdot 10^{10}$. Hence $0.02d^{3/2} \ge 0.34\cdot 10^{15}$. We claim that for these integers
 $d$ we always have $Kd^{3/2} -6\epsilon d \ge 0.02d^{3/2}$, where $K:= \frac{2}{3}\frac{1}{10}^{3/2}$ and $\epsilon:= \frac{11}{20} +4(\frac{1}{20})^{3/2}$. Indeed, $K\ge 0.021$
 and $6\epsilon \le 0.6$ and for $d\ge 31.3\cdot 10^{10}$ we have $0.001\sqrt{d} \ge 0.6$.
Hence if $g\le 0.02d^{3/2}$, then we apply \cite[Corollary 1.3]{bef}. If $g\ge 0.34\cdot 10^{15}$, then we apply Theorem \ref{i2}. \end{proof}

  \begin{proof}[Proof of Theorem \ref{i3}:]
 By Theorem \ref{i1} we may assume that $g< G_A(d,m)$. We prove Theorem \ref{i3} for
 the fixed integer $m$ by induction on the integer $d$.

  \quad (a) First assume that $(d-1,m)$ is not in the Range A, i.e. assume $d-1 \le \frac{m^2+4m+6}{6}$. Since $d\le \frac{m^2+4m+6}{6} +1$, we
 have $G_A(d,m) \le 1 + m-1 +\frac{(m-1)(m^2+4m+6)}{6} -\binom{m+2}{3} = m-1$. Since $d >  \frac{m^2+4m+6}{6}$, we have $g\le d-3$ and so it is sufficient to quote \cite{be2}
 and that in this range of degrees and genera a general non-special curve $C$ has $h^1(N_C(-2)) =0$ (Remark \ref{ov8}).

 \quad (b) Now assume that $(d-1,m)$ is in the Range A. By the inductive assumption for each integer $q$ such that $0\le q \le G_A(d-1,m)$
 there is a smooth and connected curve $Y\subset \PP^3$ of genus $q$ and degree $d-1$ such that $h^1(N_Y(-1)) =0$ and $h^0(\Ii _Y(m-1)) =0$. Let $L\subset \PP^3$
 be a general line intersecting quasi-transversally $Y$ at a unique point $p$. The vector bundle $N_{Y\cup L|Y}(-1)$ is obtained from $N_Y(-1)$ making a positive
 elementary transformation at $p$ (\cite{hh}). The vector bundle $N_{Y\cup L|L}$ is obtained from $N_L(-1)$ making a positive transformation at $p$ (\cite{hh}) and hence
 it is a direct sum of a line bundle of degree $1$ and a line bundle of degree $0$. A Mayer-Vietoris exact sequence gives $h^1(N_{Y\cup L}(-1)) =0$ and hence $Y\cup L$ is smoothable (\cite[Corollary 1.2]{fl1}). If $g\le G_A(d-1,m)$, then it is sufficient to take $q:= g$. Now assume $G_A(d-1,m)<g<G_A(d,m)$. In this range any solution $C$
 must have $h^0(\Ii _C(m-1)) =0$ and hence $h^1(\Ii _C(m-1)) =G_A(d,m)-g$.
 
 Assume for the moment that $g\ge 0.34\cdot 10^{15}$. We repeat the construction of Theorem \ref{i2} for this integer $g$. Remember that the integers $t$, $k$, $y$, $u(x,t,k)$ and $v(x,t,k)$
 only depends on $g$ and the parity of $m$. We have $1+(m-1)d -g = \binom{m+2}{3} +G_A(d,m) -g$. Since $G_A(d,m) -G_A(d-1,m) =m-1$, we have $1\le G_A(d,m) -g \le m-2$.
 Since $3-g +(m-1)u(m-1,t,k)+v(m-1,t,k) = \binom{m+2}{3}$ and $0\le v(m-1,t,k)\le m-2$, we have $u(m-1,t,k)-1 \le d \le u(m-1,t,k) +2$ and the first inequality
 holds only if $g=G_A(d,m)-1$ and $v(m-1,t,k) =0$. Hence it is sufficient to adapt $B'(m-3,t,k)$ with the new value of $d$.
 
 Now assume $g< 0.34\cdot 10^{15}$. Since $m\ge 13.8\cdot 10^5$ and $d> \frac{m^2+4m+6}{6}$, we have $d \ge 31.3\cdot 10^{10}$. Hence $0.02\cdot d^{3/2} \ge 0.34\cdot 10^{15}$. Thus as in the proof of Theorem \ref{i1} it is sufficient
 to quote \cite[Corollary 1]{bef}.
  \end{proof}
  
  \begin{remark}\label{i4}
  Fix positive integers $m, d, g$, $m\ge 3$, such that there is a smooth, connected and non-degenerate curve $C\subset \PP^3$ with degree $d$, genus $g$, $h^0(\Ii _C(m-1)) =0$
  and $h^1(N_C(-1)) =0$. The latter condition implies that asymptotically for $d\gg 0$ we are not far from the {\emph{generalized Range A}} $\frac{m^2+4m+6}{6} \le d \le D_m:= m(m+1)/2$ of \cite[Proposition 4.2]{bbem}. These conditions are satisfied with $g = G_A(d,m)$ if $m\gg 0$ and $(d,m)$ is in the Range A (Theorem \ref{i1}) or for the $(m,d,g)$ covered by \cite[Proposition 4.2]{bbem} or if $m\gg 0$ and $g\le G_A(d,m)$ (Theorem \ref{i3}).
  
  \quad {\emph{Claim:}} For each integer $d_1>d$ there is  a smooth and connected curve $X\subset \PP^3$ with degree $d_1$, genus $g$, $h^0(\Ii _X(m-1)) =0$
  and $h^1(N_X(-1)) =0$.
  
  \quad {\emph{Proof of the Claim:}} By induction on $d_1$ we reduce to the case $d_1=d+1$.
  Fix $p\in C$. Let $L\subset \PP^3$ be a general line containing $p$. We have $\sharp (C\cap L)=1$ and $L$ is not tangent to $C$ at $p$. Obviously $h^0(\Ii _{C\cup L}(m-1)) =0$. As in step (b) of the proof of Theorem \ref{i3}
we get $h^1(N_{C\cup L}(-1)) =0$. Hence $C\cup L$ is smoothable (\cite[Corollary 1.2]{fl1}). Use the semicontinuity theorem. \qed.

When $d_1\gg d$ we may cover some pairs $(d_1,g')$ with $g'>g$ taking in the proof of the Claim instead of a line a smooth rational curve $D$ of degree $d_1-d$.
with $\sharp (D\cap C)=g'-g+1$ and $D$ intersecting quasi-transversally $C$. Since any two quintuples of points of $\PP^3$ in linearly general position are projectively equivalent,
for $g'-g \le 4$ we may see $D$ as a general rational space curve of degree $d_1-d$ and hence if $d_1-d\ge 3$ we may assume that the normal bundle of $D$
is a direct sum of two line bundles of degree $2d_1-2d-1$.
  \end{remark}
  
  \begin{proof}[Proof of Corollary \ref{i5}:] By Theorem \ref{i3} there is a smooth, connected and non-degenerate curve $Y\subset \PP^3$ with degree $\delta$, genus $g$, $h^0(\Ii _Y(m-1)) =0$
  and $h^1(N_Y(-1)) =0$. Apply the Claim in Remark \ref{i4}.
  \end{proof}

 Now we discuss the weak parts of our proof of Theorem \ref{i2}, since any improvement of these parts would give huge improvements for the lower bounds assumed in Theorem \ref{i2}
 and hence for the assumption of the results in the introduction. A key point was using
 \cite[Corollary 2.4]{bbem}, since for $d$ large with respect to $m$ our proof is far less efficient (for a fixed $m$ it leaves gaps in the set of all $d$ satisfying (\ref{eqi1}).
 Hence if there is some other construction of good curves covering, say, the range $\frac{m^2+4m+6}{4.5} \le d < \frac{m^2+4m+6}{4}$, then it would be a very good help and we could use pairs $(t,k)$ with far lower $t/k$; Lemmas \ref{oov2}, \ref{oov2.1} and \ref{oov4} show how better are the bounds and simplified the proofs if it is sufficient to take all $(t,k)$ with $k\le t\le 3k$. Then (as observed at the end of Remark \ref{ov8})  a further improvement may be obtained by sharpening the results of \cite{bef}. The very first step for the latter
 project  (as observed at the end of Remark \ref{ov8}) is to use \cite{l} instead of \cite{pe}.

\end{document}